\documentclass[11pt]{amsart}
\usepackage[utf8]{inputenc}
\usepackage[T1]{fontenc}

\usepackage[english,frenchb]{babel}
\AtBeginDocument{\selectlanguage{english}}

\usepackage[final]{graphicx}

\usepackage{setspace}
\usepackage{subfigure}
\usepackage{subcaption}
\usepackage{amsmath,amsthm,amsfonts,bm}
\numberwithin{equation}{section}
\usepackage{mathrsfs,amssymb}
\usepackage{mathabx}

\usepackage{enumerate}
\usepackage{verbatim}
\usepackage{mathtools}
\usepackage{cite}
\usepackage{dsfont}
\usepackage{upgreek}
\usepackage{comment}
\usepackage[colorinlistoftodos]{todonotes}
\usepackage{amscd,tikz,tikz-cd} 
\usepackage[all]{xy}
\usepackage{graphicx}
\usepackage{dynkin-diagrams}
\usetikzlibrary{backgrounds}
\usetikzlibrary{matrix,arrows,decorations.pathmorphing}
\usepackage{rank-2-roots} 

\usepackage{float}

\usepackage{hyperref}
\hypersetup{hypertex=true,
  colorlinks=true,
  linkcolor=blue,
  anchorcolor=blue,
  citecolor=blue}

\usepackage{xcolor}
\usepackage{url}

\hypersetup{
    colorlinks=true,
    urlcolor=blue 
}

\allowdisplaybreaks[1]

\usepackage{calc}

\theoremstyle{plain}
\newtheorem{thm}{Theorem}[section]
\newtheorem*{thm*}{Theorem}
\newtheorem{conj}{CONJECTURE}[section]
\newtheorem{lem}[thm]{Lemma}
\newtheorem{prop}[thm]{Proposition}
\newtheorem{cor}[thm]{Corollary}
\newtheorem{defn-exmp}[thm]{Definition-Example}

\theoremstyle{definition}
\newtheorem{defn}[thm]{Definition}
\newtheorem{defn-prop}[thm]{Definition-Proposition}

\theoremstyle{remark}
\newtheorem*{rem}{Remark}

\newcommand{\quatrematrix}[4]{\begin{pmatrix}
#1 & #2  \\
#3 & #4
\end{pmatrix}}
\renewcommand{\refeq}[1]{(\ref{#1})}

\newcommand{\Z}{\mathbb{Z}}
\newcommand{\C}{\mathbb{C}}

\newcommand{\Res}{\operatorname{Res}}

\newcommand{\End}{\operatorname{End}}

\newcommand{\Hom}{\operatorname{Hom}}

\newcommand{\Span}{\operatorname{Span}}
\newcommand{\Id}{\operatorname{Id}}

\newcommand{\Rep}{\operatorname{Rep}}

\newcommand{\Repf}{\operatorname{Rep}_{\operatorname{f}}}

\newcommand{\smod}{\textrm{-}\operatorname{mod}}

\renewcommand{\H}{\mathcal{H}}

\newcommand{\St}{\operatorname{St}}

\newcommand{\Ind}{\operatorname{Ind}}

\newcommand{\ind}{\operatorname{ind}}
\newcommand{\res}{\operatorname{res}}

\newcommand{\Waff}{W_{\operatorname{aff}}}
\newcommand{\Saff}{S_{\operatorname{aff}}}
\newcommand{\R}{\mathcal{R}}
\newcommand{\WR}{W(\R)}
\newcommand{\NI}{\operatorname{i}}
\newcommand{\Nr}{\operatorname{r}}

\newcommand{\ND}{\operatorname{D}}

\newcommand{\GL}{\mathrm{GL}}
\newcommand{\SL}{\mathrm{SL}}

\newcommand{\SO}{\mathrm{SO}}

\newcommand*{\Unitary}{\ {}^{\scriptsize\textcolor{green!50!black}{(U)}}}

\title[The Aubert-Zelevinsky involution for $G_2$ and its associated Hecke algebras]{The Aubert-Zelevinsky involution for $G_2$ and its associated Hecke algebras}
\author{Chuan Qin}

\calclayout

\begin{document}
\begin{abstract}
Motivated by the recent work of Aubert-Xu and the techniques in G. Muic's
article, we provide examples of computations of the Aubert-Zelevinsky duality
functor for the principal and mediate series of the exceptional group $G_2$, and
deduce corresponding results regarding the involution on the Hecke algebra side.
These computations also allow us to confirm several instances of the Bernstein
conjecture for $G_2$. This article is developed from part of the author's PhD thesis.
\end{abstract}

\maketitle
\tableofcontents
\section{Introduction}
 Let $\mathbb{G}$
be the exceptional group of type $G_2$, and $G = \mathbb{G} (F)$ the
$F$-rational points. Let $T$ be a maximal $F$-split
torus of $G_2$ and $B= TU$ be a Borel subgroup. In this article, we compute the restriction of the Aubert–Zelevinsky duality functor $\ND_G$ to the subcategory $\Rep_{\mathrm{f}}^{\mathfrak{s}}(G)$  formed by finite length representations of the \emph{Bernstein block} $\Rep^{\mathfrak{s}}(G)$  (\emph{i.e.} consisting of the representations whose composition factors are subquotients of a parabolically induced representation $\NI_P^G (\sigma)$, with $\mathfrak{s} = [L, \sigma]_G$  indexed by a proper Levi subgroup $L \subset P$ of $G$ and a supercuspidal representation $\sigma$ of $L$). After finding a progenerator $\Sigma$ of $\Rep^{\mathfrak{s}_L}(L) $ ($ \mathfrak{s}_L = [L, \sigma]_L$), the equivalence of categories also yields an involution on the category of finite-dimensional right modules over the endomorphism algebra $\H^{\mathfrak{s}}(G) := \End_G(\NI^G_P(\Sigma))$ (it can be proved to be a Hecke algebra).

We adopt the notation of indexing triples and standard modules from \cite{KazhdanLusztig1987}. We provide explicit computations of the Aubert–Zelevinsky duality functor for the principal and intermediate series of the exceptional group $G_2$, and deduce the corresponding involution on the Hecke algebra side. These computations also allow us to verify numerous cases of the Bernstein conjecture for $G_2$.

Sections \ref{subsec:casebycase} to \ref{subsec:case3last} present detailed case-by-case analyses following the approach first developed in \cite{Muic1997}, which we extend to obtain further results. A summary of these results is given in Section \ref{subsec:Main}.

\subsection{Main results}\label{subsec:Main}
We list the main results in the following tables, the unitarizable ones are
labeled by $\Unitary$:

\begin{table}[H]
    \centering
    \begin{tabular}{|c|c|c|}
        \hline
        $\mathfrak{s} = [M_{\alpha}, \sigma]$ & $\Rep^s (G_2)$ & $\H^{\mathfrak{s}} \smod $ \\ \hline
        $(-)$ & $\pi (\sigma) \Unitary$ & $M_{t_a, e_{\alpha_1}, 1}$ \\ \hline
        $D(-)$ & $-J(\sigma{}) \Unitary$  & $-M_{t_a,0,1}$ \\ \hline
    \end{tabular}
     \caption{}
    \label{tab:1}
\end{table}

\begin{table}[H]
    \centering
    \begin{tabular}{|c|c|c|}
        \hline
        $\mathfrak{s} = [M_{\beta}, \sigma \otimes \tilde{\beta}]$ & $\Rep^s (G_2)$ & $\H^{\mathfrak{s}} \smod $ \\ \hline
        $(-)$ & $\pi (\sigma) \Unitary$ & $M_{t_a, e_{\alpha_1}, 1}$ \\ \hline
        $D(-)$ & $-J(\sigma{}) \Unitary$  & $-M_{t_a,0,1}$ \\ \hline
    \end{tabular}
     \caption{}
    \label{tab:2}
\end{table}

  \begin{table}[H]
    \centering
{\renewcommand\arraystretch{1.25}
  \resizebox{0.5\textwidth}{!}{\begin{tabular}{|c|c|c|}
\hline
   $\mathfrak{s}=[T, \xi \otimes 1]_G$ & \multicolumn{2}{|c|}{Case (3) $\xi$ not quadratic, $\mathfrak{J}^{\mathfrak{s}} =\GL_2(\C)$}  \\ \hline
      &  $\Rep^s (G_2)$ &  $\H^{\mathfrak{s}} \smod$ \\ \hline
    $(-) $ &  $I_\alpha\left( \delta (\nu^{\pm 1/2} \xi_2)
\right)$ & $M_{t_a, e_{\alpha},1}$\\ \hline
      $D (-)$ &  $I_\alpha\left( \nu^{\pm 1/2} \xi_2 \circ \operatorname{det}\right)$  & $M_{t_a,
    0, 1}$  \\ \hline
    \end{tabular}} }
   \caption{}
    \label{tab:5}
\end{table}

\begin{table}[H]
{\renewcommand\arraystretch{1.25}
    \centering
    \resizebox{\textwidth}{!}{ \begin{tabular}{|c|c|c|c|c|c|c|}
        \hline
      $\mathfrak{s} = [T, \xi \otimes \xi]$ & \multicolumn{2}{p{4cm}|}{Case (2) with $\xi_2$ ramified cubic, $\mathfrak{J}^{\mathfrak{s}} = \SL_3 (\C)$} & \multicolumn{2}{p{4cm}|}{Case (2) with $\xi_2$ ramified noncubic, $\mathfrak{J}^{\mathfrak{s}} = \GL_2 (\C)$}  & \multicolumn{2}{p{4cm}|}{$\xi_2$ is unramified, $s_2 \neq \pm 1$, $\mathfrak{J}^{\mathfrak{s}}= G_2 (\C)$} \\ \hline
      & $\Rep^s (G_2)$ & $\H^{\mathfrak{s}} \smod$ & $\Rep^s (G_2)$ & $\H^{\mathfrak{s}} \smod$ & $\Rep^s (G_2)$ & $\H^{\mathfrak{s}} \smod$  \\ \hline
      $(-)$    & $\pi (\nu^{\mp 1} \otimes \xi_2)$   & $M_{t_b, e_{\alpha_2},1} $ & $\pi (\nu^{\mp 1} \otimes \xi_2)$ & $M_{t_a, e_{\alpha_1},1}$&  $\pi (\nu^{\mp 1} \otimes \xi_2)$ & $M_{t_g, e_{\alpha_1},1}$ \\ \hline
       $D (-)$ &  $J(\nu^{\mp 1} \otimes \xi_2)$     & $M_{t_b,
    0, 1}$ &  $J(\nu^{\mp 1} \otimes \xi_2)$  & $M_{t_a,
    0, 1} $ & $J(\nu^{\mp 1} \otimes \xi_2)$ & $M_{t_g,
    0, 1} $ \\ \hline
    \end{tabular}} }
  \caption{}
    \label{tab:3}
  \end{table}

\begin{table}[H]
{\renewcommand\arraystretch{1.25}
  \centering
  \resizebox{1.1\textwidth}{!}{\begin{tabular}{|c|c|c|c|c|c|c|c|}
\hline
\multicolumn{4}{|c|}{ Case (3) with $\chi$ ramified quadratic, $\mathfrak{J}^{\mathfrak{s}}
=\SO_4(\C)$}   & \multicolumn{4}{c|}{Case (3) with $\chi$ ramified cubic, $\mathfrak{J}^{\mathfrak{s}}
=\SL_3(\C)$}  \\ \hline
       $\Rep^s (G_2)$ & $\H^{\mathfrak{s}} \smod$ & $\Rep^s (G_2)$ & $\H^{\mathfrak{s}} \smod$ & $\Rep^s (G_2)$ & $\H^{\mathfrak{s}} \smod$ &  $\Rep^s (G_2)$ &  $\H^{\mathfrak{s}} \smod$ \\ \hline
      $\pi (\chi) \Unitary $   & $M_{(t_a, e_{\alpha_1}, 1), (t_a, e_{\alpha_1}, 1)} $ & $J_{\alpha}
 (1/2, \delta(\chi)) \Unitary$ & $M_{t_a, e_{\alpha_1},1}$&  $\pi (\chi) \Unitary $ & $M_{t_a, e_{\alpha^{\vee}}+e_{2 \alpha^{\vee} + 3 \beta^{\vee}}, 1}$ & $J_{\alpha} (1/2, \delta(\chi))$ & $M_{t_a, e_{\alpha^{\vee}}, 1} $\\ \hline
       $J_{\beta} (1, \pi (1, \chi)) \Unitary$     & $M_{(t_a, 0, 1), (t_a, 0, 1)}$ &  $J_{\beta} (1/2, \delta (\chi)) \Unitary $  & $M_{(t_a, e_{\alpha_1}, 1), (t_a, 0, 1)}  $ & $J_{\beta} (1, \pi (\chi^{-1}, \chi^{-1})) \Unitary$ & $M_{t_a, 0,1} $ & $J_{\alpha} (1/2, \delta (\chi^{-1}))$ & $M_{t_a, e_{3 \beta^{\vee}+2 \alpha^{\vee}}, 1}$ \\ \hline
    \end{tabular}} }
   \caption*{ \ref{tab:3} Continued}
    \label{tab:3continued}
\end{table}

\begin{table}[H]
{\renewcommand\arraystretch{1.25}
    \centering
    \resizebox{1.1\textwidth}{!}{ \begin{tabular}{|c|c|c|c|c|c|c|}
        \hline
      $\mathfrak{s} = [T,    1]$ & \multicolumn{6}{|c|}{Case (2) with $\chi =1$, $s = 1/2$, $\mathfrak{J}^{\mathfrak{s}} = G_2 (\C)$} \\ \hline
      & $\Rep^s (G_2)$ & $\H^{\mathfrak{s}} \smod$ & $\Rep^s (G_2)$ & $\H^{\mathfrak{s}} \smod$ & $\Rep^s (G_2)$ & $\H^{\mathfrak{s}} \smod$  \\ \hline
      $(-)$    & $\pi (1) \Unitary $   & $M_{t_e, e_{\alpha^{\vee}} + e_{\alpha^{\vee}} + 2 \beta^{\vee}, (21)} $ & $J_{\beta} (1/2, \delta (1)) \Unitary $ & $M_{t_e, e_{\alpha^{\vee} + \beta^{\vee}}, 1}$&  $\pi ' (1) \Unitary $ & $M_{t_e, e_{\alpha^{\vee}} + e_{\alpha^{\vee}} + 2 \beta^{\vee}, (3)}$ \\ \hline
       $D (-)$ &  $J_{\alpha} (1/2, \delta (1)) \Unitary $     & $M_{t_e, e_{\alpha^{\vee}}, 1}$ &  $J_{\beta} (1/2, \delta (1)) \Unitary$  & $M_{t_e, e_{\alpha^{\vee} + \beta^{\vee}}, 1}  $ & $ J_{\beta} (1, \pi (1,1)) \Unitary $ & $M_{t_e, 0, 1} $ \\ \hline
    \end{tabular}} }
  \caption{}
    \label{tab:4}
  \end{table}

\begin{table}[H]
{\renewcommand\arraystretch{1.25}
  \centering
  \resizebox{\textwidth}{!}{\begin{tabular}{|c|c|c|c|c|c|}
\hline
\multicolumn{4}{|c|}{Case (2) with $s=3/2$ and $\chi=1$ , $\mathfrak{J}^{\mathfrak{s}} = G_2(\C)$}   & \multicolumn{2}{c|}{Case (3) with $\xi_2$ unramified,  $\mathfrak{J}^{\mathfrak{s}} = G_2(\C)$}  \\ \hline
       $\Rep^s (G_2)$ & $\H^{\mathfrak{s}} \smod$ & $\Rep^s (G_2)$ & $\H^{\mathfrak{s}} \smod$ & $\Rep^s (G_2)$ & $\H^{\mathfrak{s}} \smod$ \\ \hline
      $\St_{G_2} \Unitary $   & $M_{{t_a, e_{\alpha^{\vee}}+ e_{\beta^{\vee}} , 1 }} $ & $J_{\alpha} (3/2, \delta (1))$ & $M_{t_a, e_{\alpha^{\vee}}, 1}$&  $I_\alpha\left( \delta (\nu^{\pm 1/2} \xi_2)
\right)$ & $M_{t_g, e_{\alpha},1}$ \\ \hline
       $1_{G_2} \Unitary $     & $ M_{t_a, 0, 1}$ &  $J_{\beta} (5/2, \delta (1)) $  & $M_{t_a, e_{\beta^{\vee}}, 1} $ & $I_\alpha\left( \nu^{\pm 1/2} \xi_2 \circ \operatorname{det}\right)$ & $ M_{t_g,
    0, 1} $ \\ \hline
    \end{tabular}} }
   \caption*{\ref{tab:4} Continued}
    \label{tab:4continued1}
\end{table}

\begin{table}[H]
{\renewcommand\arraystretch{1.25}
  \centering
  \resizebox{1.1\textwidth}{!}{\begin{tabular}{|c|c|c|c|c|c|c|c|}
\hline
\multicolumn{4}{|c|}{ Case (3) with $\chi$ unramified
quadratic,  $\mathfrak{J}^{\mathfrak{s}}
=G_2(\C)$}   & \multicolumn{4}{c|}{Case (3) with $s=1/2$ and $\chi$ cubic,  $\mathfrak{J}^{\mathfrak{s}}
=G_2(\C)$}  \\ \hline
       $\Rep^s (G_2)$ & $\H^{\mathfrak{s}} \smod$ & $\Rep^s (G_2)$ & $\H^{\mathfrak{s}} \smod$ & $\Rep^s (G_2)$ & $\H^{\mathfrak{s}} \smod$ & $\Rep^s (G_2)$ & $\H^{\mathfrak{s}} \smod$ \\ \hline
      $ \pi (\chi) \Unitary $   & $M_{t_d, e_{\alpha^{\vee}}+e_{\alpha^{\vee} + 2 \beta^{\vee}}, 1}$ & $J_{\alpha}
 (1/2, \delta(\chi)) \Unitary $ & $M_{t_d, e_{\alpha^{\vee}}, 1}$&  $\pi (\chi) \Unitary $ & $M_{t_c, e_{\alpha^{\vee}}+e_{ \alpha^{\vee} + 3 \beta^{\vee}}, 1}$ & $J_{\alpha} (1/2, \delta(\chi))$ & $ M_{t_c, e_{\alpha^{\vee}}, 1}$\\ \hline
       $J_{\beta} (1, \pi (1, \chi)) \Unitary $     & $M_{t_d, 0,1}$ &  $J_{\beta} (1/2, \delta (\chi)) \Unitary $  & $M_{t_d, e_{2\beta^{\vee}+\alpha^{\vee}}, 1}  $ & $J_{\beta} (1, \pi (\chi^{-1}, \chi^{-1})) \Unitary $ & $M_{t_c, 0,1} $ & $ J_{\alpha} (1/2, \delta (\chi^{-1}))$ & $M_{t_c, e_{3 \beta^{\vee}+ \alpha^{\vee}}, 1}$ \\ \hline
    \end{tabular}} }
   \caption*{\ref{tab:4} Continued II}
    \label{tab:4continued2}
\end{table}
\subsection*{Acknowledgements.} The author thanks Anne-Marie Aubert for her
detailed explanations of \cite{AubertXu2023explicit} and
\cite{AubertXu2022Hecke}, as well as for many helpful suggestions. Gratitude is
also extended to Chi-Heng Lo for pointing out the methods in \cite{HJLLZ2024} can be applied to unitarity problems for $G_2$ and to Noriyuki Abe for identifying an ambiguity in the original proof of Proposition \ref{prop:Muic4.2}. The author also thanks Paul Boisseau, David Renard, Yi Shan, and Zhixiang Wu for valuable comments and discussions. Finally, the author gratefully acknowledges the research environment provided by Sorbonne University and the ITS at Westlake University, as well as research funding support from Professor Huayi Chen.

\section{Notation and preliminaries}\label{sec:padicgrppreliminaries}
\subsection{Representations of $p$-adic groups}
Let $H \subset G$ be a subgroup and $(\rho, V_{\rho})$ a smooth representation of $H$, we write $(\ind_{H}^G(\rho),
\ind_H^G (V_{\rho}))$ for the compactly induced representation. For $ g \in G$, ${}^g (\textrm{-})$ denote the action
$g (\textrm{-} )g^{-1}$ on $G$, ${}^g \rho$ denote the representation $\rho
({}^{g^{-1}} (\textrm{-}))$. From now on until the
end of this paragraph, we assume
$H$ is open in $G$. For a smooth $(\pi,V)$ representation of $G$,
we write $(\res^G_H (\pi), \res^G_H(V))$ for the restriction functor. We have an
adjoint pair $(\ind_H^G, \res_H^G)$ in the sense that there exists a natural
equivalence $\Hom_G (\ind_H^G (\rho), -) \cong \Hom_H (\rho ,  \res^G_H (-))$.

Let us fix a Borel subgroup $B$ of $G$ and a maximal $F$-split
torus $T$ contained in $B$. For any standard subgroup $P$, we have explicitly $P = LU$ where $U$ is the unipotent radical and we also write
the opposite Levi by $\overline{P} = L \overline{U}$. Let $\mathcal{P}(L)$
denote the set of parabolic subgroups with Levi component $L$. Let $W (G,T):=N_G (T)/T$
denote the (finite)
Weyl group of $G$, $R$ the set of roots and $S$ the set of simple roots
determined by the choice of $B$. For $I \subset S$, let $P_I$ denote the
standard parabolic $F$-subgroup of $G$ determined by $I$, and $L_I$ the Levi
subgroup of $P_I$. We write
$\widehat{T}$ for the complex torus dual to $T$, $\widehat{G}$ for the Langlands
dual group of $G$, $W_F$ for the Weil group of
$F$, $I_F$ for the inertial group of $F$.

For a smooth  $(\rho, V_{\rho})$ representation of a Levi subgroup $L \subset P=LU$, $\rho$ extends to a
representation of $P$ by letting $U$ act trivially. Let $(\Ind_P^G (\rho),
\Ind_P^G (V_{\rho}))$ denote the un-normalized induction. For a smooth
representation $(\pi, V)$ of $G$, let $(\pi_U, V_U)$ denote the un-normalized
Jaquet module of $(\pi, V)$ and $j_U: V \rightarrow V_U$ denote the quotient map. We denote the normalized
induction and normalized restriction (Jacquet module) of $(\rho, V_{\rho})$ by
$(\NI_P^G (\rho), \NI_P^G (V_{\rho}) )$ and $(\Nr_P^G (\rho), \Nr_P^G
(V_{\rho}))$ (or $(\Nr_U(\rho), \Nr_U(V_{\rho}))$) respectively. We have two
adjoint pairs: $(r_U, \NI_P^G)$, where there exists a natural equivalence $\Hom_G ( - , \NI_P^G (\rho)) \cong \Hom_L (r_U (-),
\rho)$ and $(\NI_P^G, r_{\overline{U}})$ with  $\Hom_G ( \NI_P^G (\rho) , - ) \cong \Hom_L (\rho,
 r_{\overline{U}}(-))$ (the second adjoint theorem, see \cite[VI.9]{Renard2010}).

Let $M$ be a Levi subgroup of $G$, $\Rep (M)$ denote the category of all smooth
complex representations of $M$,  $\Repf (M)$ denote the subcategory of finite
length representations, and $\operatorname{Irr}(M)$ denote the set of isomorphism classes of smooth
irreducible representations. Let $\mathfrak{R}(G)$ be the Grothendieck group of
$\Repf (G)$ and for $\pi \in \Repf (G)$, denote by $[\pi]$ its image in $\mathfrak{R}(G)$.

\subsection{The Aubert-Zelevinsky duality}\label{AZduality}
  For $\pi \in \Repf (G)$ we follow \cite{Aubert1995}, consider the map
\begin{equation}\label{eq:AZDuality}
  \begin{aligned}
    \ND_G: \mathfrak{R}(G) & \rightarrow  \mathfrak{R}(G)  \\
                      [\pi]   &  \mapsto [\sum_{I \subset S} (-1)^{|I|} \NI_{P_I}^{G} \circ \Nr^G_{P_I} (\pi) ].
  \end{aligned}
\end{equation}
It satisfies the following basic properties which are proved by \cite[Theorem 1.7]{Aubert1995}:

  \begin{enumerate}[(1)]
  \item We have $\ND_{G} ((-)^{\vee}) =  (\ND_G(-))^{\vee}$ where $^{\vee}$ means taking contragredient.
    \item For any subset $J \subset S$, $\ND_G \circ \NI_{P_J}^{G} = \NI_{P_J}^G
      \circ \ND_{L_J}$.
    \item $\ND_G^2 = \Id$.
    \item If $\pi$ is irreducible cuspidal, then $\ND_G (\pi) = (-1)^{|S|} \pi$.
    \end{enumerate}

\subsection{The Bernstein decomposition}

We write $[L, \sigma]_{G}$ for the inertial equivalence class of $(L, \sigma)$
(sometimes omitting the sub-index if no confusion is caused) and
$\mathfrak{B}(G)$ for the set of all \emph{inertial equivalence classes} where the equivalence relation is given by  $(L_1 , \rho_1) \sim (L_2, \rho_2)$ if there exists $g \in G$, such that $gL_1g^{-1} = L_2$ and $\rho_2 = \rho_1^{g} \otimes \omega $ for some unramified character $\omega$ of $L_2$). If $M$ is a
proper Levi subgroup of $G$, $L$ is a Levi subgroup of $M$ and $\sigma$ is a
supercuspidal representation of $L$, then such $\mathfrak{s}_M = [L, \sigma]_M \in
\mathfrak{B}(M)$ determines a $\mathfrak{s}_G=[L,\sigma]_G \in
\mathfrak{B}(G)$ naturally. If $(\pi ,
V)$ is an irreducible smooth representation of $G$, one defines the \emph{inertial support}
$\mathfrak{I}\left(\pi \right)$ of $(\pi, V)$ to be the inertial equivalence class of the
supercuspidal support of $\pi$. The subcategories $ \Rep^{\mathfrak{s}}(G) , \ \mathfrak{s} \in \mathfrak{B}(G)$
split the category $\Rep (G)$, and we have the Bernstein decomposition of
$\Rep (G)$ (see \cite{Bernstein1984}):
\begin{equation}
  \label{eq:BernsteinDecomposition}
  \Rep (G)=\prod_{\mathfrak{s} \in \mathfrak{B}(G)} \Rep^{\mathfrak{s}}(G)
\end{equation}
of the subcategories $\Rep^{\mathfrak{s}}(G)$ as $\mathfrak{s}$ ranges through
$\mathfrak{B}(G)$. Let
\begin{equation}
  \label{eq:prpadic}
  \operatorname{pr}_G^{\mathfrak{s}}: \Rep (G) \rightarrow
  \Rep^{\mathfrak{s}}(G)
\end{equation}
denote the projection functor.

For $\mathfrak{s}=[L, \sigma]_G \in \mathfrak{B}(G)$, following
\cite[section 3.2]{AubertXu2023explicit} we define $
\mathrm{N}_G\left(\mathfrak{s}_L\right):=\left\{g \in G:{ }^g L=L\right.$ and ${
}^g \sigma \simeq \chi \otimes \sigma$, for some $\left.\chi \in
  \mathfrak{X}_{\mathrm{nr}}(L)\right\}$, where $\mathfrak{s}_L=[L, \sigma]_L
\in \mathfrak{B}(L)$ and denote by $W_G^{\mathfrak{s}}$ the extended finite Weyl
group $\mathrm{N}_G\left(\mathfrak{s}_L\right) / L$ and $\Waff^{\mathfrak{s}}$
the corresponding affine version.

\subsection{Hecke algebras associated with Bernstein blocks}

Let $K$ be a compact open subgroup of $G$ and $(\rho, V_{\rho})$ a
representation of $K$, now recall the definition of Hecke
algebra associated with $(K, \rho)$.
\begin{defn}\label{defn:HGrho}
  Let $\H (G, \rho)$ denote the space of
compactly supported functions
\[
\phi : G \rightarrow \End_{\C} (V_{\rho})
\]
satisfying
\[
\phi (k_1 gk_2) = \rho (k_1 ) \circ \phi (g) \circ \rho (k_2)
\]
for all $k_1$, $k_2 \in K$ and $g \in G$.
For $\phi_i$, $\phi_2 \in \H(G, \rho)$, the standard convolution product
\[
(\phi_{1} * \phi_2) (x) : = \int_G \phi_1(y) \circ \phi_2 (y^{-1}x) dy
\]
gives $\H (G, \rho)$ a structure of $\C$-algebra.
\end{defn}

There exists an anti-isomorphism from $\H(G, \rho)$ to $\H(G, \rho^{\vee})$ by
sending $\phi (\textrm{-})$ to $\phi ( (\textrm{-})^{-1})^{\vee}$. Our
definition for $\H(G, \rho)$ differs from
\cite[section 7]{Roche1998} and \cite[2.4]{BushnellKutzko1998} by a
contragredient, following \emph{loc.cit} 2.6 we have an algebra isomorphism
\begin{equation}
  \label{eq:HGrEndind}
  \H (G, \rho) \cong \End_{G} (\ind_K^G (\rho)).
\end{equation}
where $\End_G  (\ind_K^G (\rho))$ denotes the right $G$-endomorphism.

\subsection{Principal blocks}\label{subsec:Principalblocks}
We recall here some results from \cite{Roche1998} section 6 to 8. Let $T^1$ be
the unique maximal compact subgroup of $T$ and $\chi: T^1 \rightarrow \C$ be a
smooth character. Let $\widetilde{\chi}$ be any character of $T$ extending
$\chi$, the inertial equivalence class $\mathfrak{s}_{\chi} : =[T,
\widetilde{\chi}]$ depends only on $\chi$ (it is well-defined). An
\emph{$\mathfrak{s}$-type} in $G$ is a pair $(K, \rho)$ where $K$ is a compact open
subgroup of $G$ and $\rho$ is an irreducible smooth representation of $K$ such
that an irreducible representation $\pi$ of $G$ contains $\rho$ if and only if
$\mathfrak{I}(\pi) = \mathfrak{s}$. In \cite{Roche1998}, A. Roche constructed an $\mathfrak{s}_{\chi}$-type $(J, \rho)$ where $J$ is a
compact open subgroup of $G$ and $\rho$ is a smooth character of $J$ such that
$J \cap T = T^1$ and $\rho|_{J \cap T} = \chi$. Recall Definition
\ref{defn:HGrho}, $ \H (G, \rho)$ is the space of complex compactly supported
functions equipped with a star operation $\kappa$ given by
\[
\kappa (f)(x) = \overline{f(x^{-1})}, \textrm{ for }f \in \H (G, \rho).
\]
 There is an equivalence of
categories $\Rep^{\mathfrak{s_{\chi}}} (G) \rightarrow \H (G, \rho) \textrm{-}
\operatorname{Mod}$ by \emph{loc.cit.} Theorem 7.5 and Corollary 7.9.

There exists a dual group interpretation of $\H (G, \rho)$. Applying the Local Langlands
correspondence for $T$, we have a homomorphism $\varphi_{\widetilde{\chi}}: W_F
\rightarrow \widehat{T}$ associated with $\widetilde{\chi}: T \rightarrow \C$. By
local class field theory, $\varphi_{\widetilde{\chi}}$ restricted to $I_F$ depends only on
$\chi$ and we denote it by $\varphi_{\chi}$. Let $\mathfrak{J}^{\mathfrak{s}}$ denote the centralizer in $\widehat{G}$ of the
image of $\varphi_{\chi}$. Since $\mathbb{G}$ has connected center, the group
$\mathfrak{J}^{\mathfrak{s}}$ is connected and its Weyl group is isomorphic to
the group $W^{\mathfrak{s}}$. Let $J_{\mathfrak{s}}$ denote the group of
$F$-rational points of the $F$-split reductive group whose Langlands dual is the
group $\mathfrak{J}^{\mathfrak{s}}$. By \cite[section 8]{Roche1998}, the algebra
$\H (G, \rho )$ and $\H (J_{\mathfrak{s}}, 1_J)$ are isomorphic via a family of
$ \kappa $-preserving, support-preserving isomorphisms.

We recall here some results from \cite{KazhdanLusztig1987} about indexing triples and standard modules for affine Hecke algebras. Let $\mathcal{G}$ be a simple complex algebraic group with root system $R$ and weight
lattice $X$. Recall that $q$ is the number of elements in the residue field of $F$. \emph{An indexing triple} $(s, n, \rho)$ consists of a semisimple element
$s \in \mathcal{G}$, a nilpotent element $n \in \operatorname{Lie}(\mathcal{G})$ such
that $\operatorname{Ad} (s) n = q n$ and an irreducible representation of the
component group $A(s, n):= Z_{\mathcal{G}} (s,n)/ Z_{\mathcal{G}}(s, n)^{\circ}$ where $Z_{\mathcal{G}} (s,n): =
Z_{\mathcal{G}}(s) \cap Z_{\mathcal{G}} (n)$. Let $n = \operatorname{exp} (n)$ denote the corresponding
unipotent element of $\mathcal{G}$. Let $K(\mathcal{B}_{s,u})$ be the $K$-theory of the
variety $\mathcal{B}_{s,u}$ where $\mathcal{B}_{s,u}$ is the variety of Borel
subgroups of $\mathcal{G}$ containing both $s$ and $u$. The group $A(s, u)$ and $\H (G, \rho )$ act
on $K (\mathcal{B}_{s,u})$ and their actions commute. The standard modules
$M_{s,n, \rho}$ are the $\H (G, \rho )$-modules in the decomposition:
\begin{equation}
  \label{eq:standardmodules}
   K( \mathcal{B}_{s,u}) = \bigoplus_{\rho \in \operatorname{Irr} (A(s,u))}
  M_{s,n, \rho} \otimes \rho .
\end{equation}

\section{Computations of the duality for Bernstein blocks of $G_2$}\label{sec:computationsG2}
\subsection{Preliminaries on $G_2$}
Let $V$ be the hyperplane of $ \mathbb{R}^3 =
\Span_{\mathbb{R}} (e_1, e_2, e_3)$
formed by points whose sum of coordinates is zero.

~\\
\begin{figure}[htbp]
  \centering
  \includegraphics[width=0.4\textwidth]{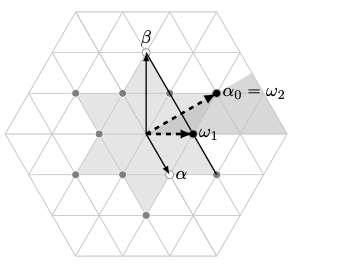}
  \label{pic:G2}
\end{figure}

The choice of simple roots $\alpha = \dfrac{1}{\sqrt{3}} (e_1-e_2)$ (\emph{the short
root}) and $ \beta =
\dfrac{1}{\sqrt{3}} (-2e_1 + e_2 + e_3)$ (\emph{the long root}) define a basis
for the root system $R$ of $G$, and $R^{+}= \{ \alpha, \beta, \alpha+\beta, 2
\alpha+\beta, 3 \alpha+\beta, 3 \alpha+2 \beta   \}$.  The dominant
weights are $\omega_1= 2\alpha +
\beta = \dfrac{1}{\sqrt{3}}(- e_2+e_3) $ and $\omega_2 =3\alpha + 2
\beta =\dfrac{1}{\sqrt{3}} (-e_1-e_2+2e_3)$. For any $\gamma \in R$, we denote
by $\gamma^{\vee}$ the corresponding coroot, and $s_{\gamma}$ the reflection in  the Weyl group $W$ of $G$ defined by $\gamma$. Let $B = TU$ be the corresponding Borel subgroup,
$P_{\gamma}$ be the maximal standard parabolic subgroup of $G$ generated by $B$
and $x_{\pm \gamma} (F)$ (defined in \cite[Section 4]{AubertXu2023explicit}).

We need to clarify some notations and definitions on Levi subgroups
of $G_2$, following in A-M. Aubert \& Y. Xu \cite{AubertXu2023explicit} and G. Muic \cite{Muic1997}, we define:
\begin{equation}\label{eq:Torusab}
  \begin{aligned}
    (t_1, t_2) \in T \cong F^{\times} \times F^{\times} & \mapsto \quatrematrix{t_{1}}{{}}{{}}{t_{2}} \in T_{\alpha} \\
    (t_1, t_2) \in T \cong F^{\times} \times F^{\times} & \mapsto \quatrematrix{t_{2}}{{}}{{}}{t_{1}t_{2}^{-1}} \in T_{\beta} \\
  \end{aligned}
\end{equation}
as the isomorphisms between the split torus of maximal Levis associated with $\{
\alpha , \beta\}$ and $T
$.

Accordingly, if $\chi_1 \otimes \chi_2$ is a character of $F^{\times} \times
F^{\times}$, then:
\begin{equation}\label{eq:Characterab}
  \begin{aligned}
    \chi_1 \otimes \chi_2  & \mapsto \quatrematrix{\chi_{1}}{{}}{{}}{\chi_{2}} \textrm{ as character of }T_{\alpha}\\
   \chi_1 \otimes \chi_2  & \mapsto \quatrematrix{\chi_1 \chi_2}{{}}{{}}{\chi_1} \textrm{ as character of }T_{\beta}
  \end{aligned}
\end{equation}
\begin{rem}
  The notations mean the following:
  \[
    \quatrematrix{\chi_{1}}{{}}{{}}{\chi_{2}}
    \quatrematrix{t_{1}}{{}}{{}}{t_{2}} = \chi_1 (t_1) \chi_2 (t_2) =  \quatrematrix{\chi_1 \chi_2}{{}}{{}}{\chi_1}\quatrematrix{t_{2}}{{}}{{}}{t_{1}t_{2}^{-1}}.
  \]
\end{rem}
The above isomorphisms between $T$ and $T_{\gamma}$ extends to isomorphisms
between $\GL_2 (F) $ and $L_{\gamma}$, for $\gamma = \alpha$ or $\beta$. The Weyl group action of characters on $F^{\times}
\times F^{\times}$ is described in \cite[Table 1]{AubertXu2023explicit}.

\subsubsection{Some definitions and abbreviations}
We will use the following fact implicitly: for a smooth representation $\sigma$ of a Levi subgroup $L \subset
P=LU$, $\sigma$ extends to a representation of $P$ by letting $U$ act trivially
(and is thus seen as a representation of $P$). We use $\nu$ to denote the normalized absolute value of $F$.
The following functors are abbreviated as follows (with $\gamma$
equals $ \alpha \textrm{ or } \beta$): the normalized parabolic induction $\NI_{T (U \cap
  L_{\gamma})}^{L_{\gamma}} $  to the intermediate Levi $L_{\gamma}$  is abbreviated as $I^{\gamma} $;  the
normalized Jacquet functor $\Nr_{P_{\gamma}}^{G_2}$ from $G_2$ to  $P_{\gamma}$  is
abbreviated as $r_{\gamma}$. The normalized parabolic induction $\NI_{B}^{G} $    is abbreviated as $I$; the
normalized Jacquet functor $\Nr_{B}^{G_2}$ from $G_2$ to  $B$  is
abbreviated as $r_{\emptyset}$.  We have
$I^{\gamma} \left( s_{\gamma}(\chi_1 \otimes
\chi_2) \right) =I^{\gamma} (\chi_1 \otimes \chi_2)$ in $\operatorname{R}(L_{\gamma})$ where
$\gamma = \alpha, \ \beta$ and $s_{\gamma}$ the corresponding simple reflection
in the Weyl group. For an admissible representation $\pi$ of
$L_{\gamma}$, $s \in \C$ and $\eta \in X_{*} (L_{\gamma}) \otimes_{\Z}
\mathbb{R}$ (where $X_{*} (L_{\gamma})$ is the group of $F$-rational characters of $L_{\gamma}$), we write
\[
I_{\gamma} (\chi, \pi) := \NI_{P_{\gamma}}^{G} ((\nu \circ \chi ) \otimes \pi ),
\]
and
\[
I_{\gamma} (s, \pi):= \NI_{P_{\gamma}}^{G} ( (\nu^{s} \circ \det) \otimes \pi ),
\ I_{\gamma} (\pi):= I_{\gamma} (0, \pi).
\]

 When $\pi$ is a tempered irreducible representation and $s \in \mathbb{R}^{>0}$,
the representation $I^{\gamma}(s, \pi)$ has a unique irreducible quotient
denoted by $J_{\gamma}(s, \pi)$ by the Langlands quotient Theorem(see
\cite[Theorem 3.5]{Konno2003}).

We end an equation with $\
[\operatorname{R}(G)]$ to mean that this is an equation in the Grothendieck
group of $G$.

\subsubsection{Representation theory of $L_{\gamma}$, $\gamma = \alpha$ or $\beta$}\label{subsubsec:GL2Rep}
We recall some $\GL(2)$ theory since the Levi factors of maximal parabolic
subgroups of $G_2$ are isomorphic to $\GL_2 (F)$. (For this subsection only) Denote by $B=TU$ a Borel
subgroup of $\GL_2 (F)$ with maximal torus $T$. For a smooth character $\chi$ of $F^{\times}$ and any smooth admissible representation $\pi$ of
$\GL_2 (F)$, we denote by $ \pi \chi$ the twist of $\pi$ by $\chi \circ \det $.
We define the \emph{generalized Steinberg representation} $\delta(\chi)$ as the unique irreducible subrepresentation of
$\NI^{\GL_2}_B (\nu ^{1 / 2} \chi \otimes \nu^{-1 / 2} \chi )$, and the
\emph{generalized trivial representation} $ \chi \circ \det$ as the unique irreducible quotient of
$\NI^{\GL_2}_B (\nu ^{1 / 2} \chi \otimes \nu^{-1 / 2} \chi )$. For unitary
characters $\chi_1, \chi_2$, denote by $\pi\left(\chi_1,
  \chi_2\right):=I^{\gamma} (\chi_1
\otimes \chi_2 )$ for $\gamma = \alpha $ or $\beta$, which  is a tempered irreducible representation. Now we recall well-known
facts about principal series representations of Levi factors of maximal
parabolic subgroups of $G_2$ following \cite[Proposition 1.1]{Muic1997}.

\begin{prop}\label{prop:Mgamma}
  Suppose that $\chi, \chi_1$ and $\chi_2$ are characters of $F^{\times}$, and
  $\gamma \in$ $\{\alpha, \beta\}$. Then we have the following.
  \begin{enumerate}[(1)]
  \item The principal series $I^{\gamma} \left(\chi_1 \otimes \chi_2\right)$ of
    $L_\gamma$ reduces if and only if $\left(\chi_1 \otimes \chi_2\right) \circ
    \gamma^{\vee}=\nu^{ \pm 1}$. If $I^{\gamma} \left(\chi_1 \otimes \chi_2\right)$
    is irreducible, it is isomorphic to $I^{\gamma} \left(s_\gamma\left(\chi_1 \otimes
        \chi_2\right)\right)$.
    \item The principal series $I^\alpha\left( \nu^{1 / 2} \chi \otimes \nu^{-1 / 2}
        \chi\right)$ (\emph{resp.} $I^\alpha\left( \nu^{-1 / 2} \chi \otimes \nu^{1 / 2}
          \chi\right)$) contains $\delta(\chi)$ and $\chi \circ$ det as a
      unique irreducible subrepresentation (\emph{resp.} quotient) and quotient
      (\emph{resp.} subrepresentation).
      \item  The principal series $I^\beta\left(\nu^{-1 / 2} \chi \otimes  \nu
        \right) $ ( \emph{resp.} $I^\beta\left(\nu^{1 / 2} \chi \otimes
          \nu^{-1}\right)$) contains $\delta(\chi)$ and $\chi \circ \det$  as a
        unique irreducible subrepresentation (\emph{resp.} quotient) and
        quotient (\emph{resp.} subrepresentation).
  \end{enumerate}
\end{prop}

The restriction operators applied to the generalized Steinberg and generalized trivial representations are
also adapted depending on $\gamma = \alpha$ or $\beta$, we have
\[
  r_T^{L_{\beta}} \left( \nu^s \delta (\chi) \right) =
  \nu^{s-1/2} \chi \otimes \nu,
\]
and
\[
  r_T^{L_{\beta}} (\nu^s \chi \circ \det)=
  \nu^{s+1/2} \chi \otimes \nu^{-1}.
\]
The tensor operation also needs to be
adapted depending on $\gamma$, we have
\[
\nu^s \otimes I^{\alpha}
(\chi_1 \otimes \chi_2)= I^{\alpha} (\nu^s \chi_1 \otimes \nu^s \chi_2),
\]
and
\[
  \nu^s \otimes
  I^{\beta} (\chi_1 \otimes \chi_2) = I^{\beta}(\chi_1 \nu^s \otimes \chi_2).
\]

\subsubsection{Involution for affine Hecke algebras}
For any element $w \in W$, we introduce its length $ \ell (w)$ as
 the smallest number  such that $w$ can be written as a product of $ \ell (w)$
 simple reflections, and $ \operatorname{sgn} (w): = (-1)^{\ell (w)}$. Let $w_0$
 denote the
 longest element of $W$.  Let $W_I$ denote the subgroup of $W$ generated by $I$. Also let $\Res_{W_I}^W$
(\emph{resp.} $\Ind_{W_I}^W$) denote the restriction (\emph{resp.} induction functor) for
representations of the finite groups $W$ and $W_I$. Let $\R$ denote the \emph{root datum} quadruple $(X, R,X^{\vee},R^{\vee})$, where  $X$ ($X^{\vee}$)  is weight (\emph{resp.} coweight), ($R^{\vee}$) $R$ is root (\emph{resp.} coroot).  Let
us recall the definition of the affine Weyl group $\Waff := W \ltimes \Z R$ (\emph{resp.} extended
affine Weyl group $ W(\R) := W \ltimes X$) as the semi-direct product of finite
Weyl group $W$ by the lattice $\Z R$ (\emph{resp.} $X$). Let $w
=w_{\mathrm{fin}} t_{\lambda}$ be the decomposition for $w$ as an element in $ W(\R)$. Here
$w_{\mathrm{fin}}$ denotes the finite Weyl group part and $t_{\lambda}$, $\lambda \in X$ denotes the
translation part
associated with some $\lambda \in X$. We know that $(\Waff, \Saff)$ is a Coxeter
system with $\Waff$
and the set of generators $\Saff := S \bigcup
\ \left\{  \textrm{an affine reflection }s_0 \right\}$. There exists a semi-direct
product decomposition of $W(\R)$ as $\Omega \ltimes
\Waff$ where $\Omega$ is the stabilizer of ``fundamental alcove'' determined by $\Saff$. Let $\ell$ be the length
function for the Coxeter system $(\Waff, \Saff)$: it extends naturally to
$W(\R)$ by requiring $\ell (w) =0$ if $w \in \Omega$. We may thus write $w =  \gamma
\tau$, with $\gamma \in \Omega$ of length $0$ and $\tau$ a product of $\ell(w)$
elements of $\Saff$.
We recall the definition for affine Hecke algebra.
\begin{defn}[Iwahori-Matsumoto]\label{defn:HeckeIM}
  For $q_s \in \mathbb{C}^{\times} \backslash \{ \textrm{roots of unity}\}$, define the
  extended affine Hecke algebra $\H= \H (\WR, q_s)$ as
  the $\mathbb{C}$-algebra with basis $T_w$, $w \in \WR$ satisfying:
  \begin{equation}
    \label{eq:HeckeAlgDef0}
    (T_s+1)(T_s -q_s) =0 \textrm{, for } s \in \Saff ,
  \end{equation}
  \begin{equation}
    \label{eq:HeckeAlgDef}
    T_w \cdot T_{w'} =T_{ww'} \textrm{, if }\ell(w) + \ell(w') = \ell(ww').
  \end{equation}
\end{defn}

We denote by $\H$ as the extended affine Hecke algebra $\H (W(\R),q_s)$ over a fixed field $K$ associated with $W(\R)$
 and unequal parameters $q_s \in \mathbb{C}^{\times}
\backslash \{ \textrm{roots of unity}\}$, $s \in S$. We also let $\H_{I}$ be the
subalgebra of $\H$ associated with $W(\R)_I: = W_I \ltimes X$.

We define the dual of a $\H$-module $(\pi, M)$ by
\[
  D[M]: = \sum_{I \subset S} (-1)^{|I|}[\Ind_I (\Res_I M)],
\]
where  $\Res_I$ is the usual restriction functor to $\H_I$, while the induction functor
$\Ind_I $ is defined by $\Ind_I N: =\H \otimes_{\H_{I}}N$ for an $\H_I$-module
$N$ and $[M]$ denotes the
image of $M$ in the Grothendieck
group of finite dimensional modules over $K$.

From \cite{Kato1993} or the author's thesis, we have
\begin{thm}\label{thm:KatoEn}
  Define a twisted action $*$ on $\H$ as follows: for \[
    w = w_{\mathrm{fin}} t_{\lambda}= w_{\Omega} t_{\mu} s_{i_1}
    s_{i_2} \cdots s_{i_r} \in \WR
  \]
set
\[
  T_w^{*} = (-1)^{l( w_{\mathrm{fin}}
    )} q(w) T_{w^{-1}}^{-1}
\]
where  $q(w) = \displaystyle \prod_{k=1}^r q_{s_{i_k}}$. Given $\H$-module $(\pi, M)$, let
$(\pi^{*},M^{*})$ be the $\H$-module such that $M^{*} = M$ as $K$-vector space, equipped
with the $\H$-action $\pi^{*}(h)(m) := \pi (h^{*}) (m)$, $\forall m \in M$ and
$\forall h \in \H$. Then we have the following equality in the Grothendieck group:
\[
D[M] = [M^{*}].
\]
\end{thm}

\subsection{Case by case discussion}\label{subsec:casebycase}

Let $\mathfrak{s} = \lbrack L, \sigma \rbrack_G $, where $L$ is a proper Levi
subgroup of $G_2$. Let $\sigma$ be an irreducible supercuspidal
representation of $L$. The Levi subgroups of $G_2$ are isomorphic to either
$\GL_2(F)$ or $F^{\times} \times F^{\times}$.

\subsection{The intermediate case $\mathfrak{s} = [L_{\alpha}, \sigma]_G$ and $
  \mathfrak{s}=[L_{\beta}, \sigma]_G$}
Let $\omega_{\sigma}$ denote the central character of $\sigma$. If $L = L_\alpha$ by \cite[Proposition 6.2]{Shahidi1989} and \cite[Section 8]{AubertXu2023explicit}:
\begin{enumerate}[(1)]
\item When $\omega_\sigma \neq 1, \NI_P^G(\sigma)$ is reducible and there
  are no complementary series.
  \item When $\omega_\sigma=1, \NI_P^G(\sigma)$ is irreducible, and $ \NI_P^G\left(\sigma \otimes \nu^s\right)$ is irreducible unless $s= \pm 1 / 2$.
\end{enumerate}
When  $\NI_P^G\left(\sigma \otimes \nu^{1 / 2}\right)$  is reducible, it has
has length 2: it has a unique generic discrete series subrepresentation
$\pi(\sigma)$ and a unique irreducible pre-unitary non-tempered Langlands
quotient, $J(\sigma)$. From \cite[4.1.2.]{AubertXu2022Hecke}, we know the Hecke algebra $\H^{\mathfrak{s}}$ in the short root case is either
\[
  \H\left(\left\{1, s_{3 \alpha+2 \beta}\right\}, q_F\right) \ltimes
  \C [\mathcal{O}]
\]
($\left\{1, s_{3 \alpha+2 \beta}\right\}$ is the group of two elements, and $q_F = $number of elements of the residue field) or
\[
  \H \left(\left\{1, s_{3 \alpha+2 \beta}\right\}, 1\right)
  \ltimes \C [\mathcal{O}].
  \]
\vspace{0.6cm}\\
If $M =L_{\beta}$, from \emph{loc.cit.} 4.1.1. there are four
possibilities for the series and Hecke algebra depending on the Plancherel measure. The
reducible case is $I_{\beta} ( \tilde{\beta} , \sigma)$ with $\tilde{\beta} =
\langle \rho , \beta \rangle^{-1} \rho $ where $\rho $ is half sum of positive
roots. In this case $I_{\beta} (\tilde{\beta}, \sigma )$ is length $2$ with a
unique $\chi$-generic subrepresentation $\pi (\sigma)$ (\emph{i.e.} it can be realized on a space of smooth
functions $W$ satisfying $W (um) = \chi (u) W(m)$, for $m \in L_{\beta}$, $u \in
U_{\beta}$, and $\chi$ is some generic character of $U_{\beta}$)  which is in the discrete
series and a preunitary non-tempered quotient $J(\sigma)$. From \cite[4.1.1.]{AubertXu2022Hecke}, we know the Hecke algebra $\H^{\mathfrak{s}}$ in the long root case are
\[
  \H\left(\left\{1, s_{2 \alpha+ \beta}\right\}, q_s \right) \ltimes
  \C [\mathcal{O}]
\]
where the parameters have four cases to discuss depending on $\omega$.

In both the short root case and the long root case, we have
\[
  \ND_G (\pi
  (\sigma)) = - J(\sigma).
\]
From now on, we will use the notations of indexing triples and
standard modules \refeq{eq:standardmodules} (see \cite{KazhdanLusztig1987} for
more details) to denote the irreducible
modules of Hecke algebras. From \cite[table 11]{AubertXu2023explicit}, we know
the involution for the modules of Hecke algebra is
\[
D_{\H^{\mathfrak{s}}} ([M_{t_a, e_{\alpha_1}, 1}]) = - M_{t_{a},0,1}.
\]

\subsection{Principal series}
We now deal with principal series. We start by recalling the setup in \cite[Section 9.3]{AubertXu2023explicit}. We
may write $\sigma = \xi_1 \otimes \xi_2 = \chi_1 \xi \otimes \chi_2 \xi$, with $\xi$ a ramified
character of $F^{\times}$, and unramified characters $\chi_1 = \nu^{s_1}$ and $\chi_2 = \nu^{s_2}$ for
$s_1$, $s_2 \in \C$.
We start with the assumption that at least one of  $ \chi_1 \xi $ and $ \chi_2
\xi$ is non-unitary, then following \cite[Proposition 3.1]{Muic1997}, $I(\nu^{s_1} \xi_1 \otimes \nu^{s_2} \xi_2)$ is irreducible unless:

  \begin{enumerate}[(1)]
\item $s_{1} \pm 1$, $ \xi =1$, $s_2$ arbitrary. (\emph{resp.} $s_2 \pm 1$, $
  \xi =1$, $s_1$ arbitrary ) \label{Muic1997itm:1}
\item $\nu^{s_1 + s_2} \xi^2 =\nu^{\pm 1}$, \label{Muic1997itm:2}
\item $s_1 - s_2 = \pm 1$, \label{Muic1997itm:3}
\item $\nu^{2s_1 +s_2} \xi^3 = \nu^{\pm 1}$, \label{Muic1997itm:4}
  \item $\nu^{s_1 + 2s_2} \xi^3 = \nu^{\pm 1}$. \label{Muic1997itm:5}
  \end{enumerate}
  The case \refeq{Muic1997itm:1} and \refeq{Muic1997itm:2} are in fact
  equivalent, \refeq{Muic1997itm:3} and \refeq{Muic1997itm:4} are equivalent.
\subsubsection{Case  \refeq{Muic1997itm:2} within the case $\mathfrak{s} = \lbrack T, \ \xi \otimes \xi \rbrack_G $}
In case \refeq{Muic1997itm:2}, we have $\nu^{s_1}_F \xi = \nu^{-s_2 \pm 1}
\xi^{-1}$. By \cite[Lemma 5.4 (iii)]{BernsteinDeligneKazhdan1986}, $I(\xi_1
\otimes \xi_2) = I( \nu^{-s_2 \pm 1} \xi^{-1} \otimes \nu^{s_2} \xi)
= I({}^{w} (\nu^{-s_2 \pm 1} \xi^{-1} \otimes \nu^{s_2} \xi) ) =
I(\nu^{\mp 1} \otimes \xi_2) \  [\operatorname{R}(G)]$, where $w = s_{\alpha} s_{\beta} s_{\alpha}  s_{\beta} s_{\alpha}$.

For the principal series case, by Section \ref{subsec:Principalblocks}, we have the Hecke
algebras $\H^{\mathfrak{s}} = \H (J_{\mathfrak{s}}, 1_J)$, we only need to find the
$J_{\mathfrak{s}}$ or their Langlands dual groups.

Now we discuss the cases when $I^{\alpha}(\nu^{\mp 1} \otimes \xi_2 \nu)$ and
  $I^{\beta} (\nu^{\mp 1} \otimes \xi_2 \nu)$ are
  irreducible. If $\nu^{\mp 1} \otimes \xi_2 \notin \{ \nu \otimes 1, \ \nu \otimes
\nu^{2}, \ \nu^{-1} \otimes 1 , \ \nu^{-1} \otimes \nu^{-2}, \
\nu^2 \otimes \nu^{-1}, \ 1 \otimes \nu^{-1}, \ 1 \otimes \nu , \ \nu^{-2}
\otimes \nu
\}$, then $I^{\alpha} (\nu^{\mp 1 -s} \otimes \xi_2 \nu^{-s})$ and
$I^{\beta} (\nu^{\mp 1} \otimes \xi_2 \nu)$ are
irreducible, but $I(\nu^{\pm 1} \otimes \xi_2)$ is reducible.

\begin{lem}
  \label{lem:raInuFmp1xi2}
  Under the above assumption, applying $r_{\alpha}$ to  $I(\nu^{\mp 1} \otimes
  \xi_2)$, we have the following
  \begin{equation} \label{eq:raInuFmp1xi2}
    \begin{aligned}
      r_{\alpha}I(\nu^{\mp 1} \otimes \xi_2) &= I^{\alpha} (\nu^{\mp 1} \otimes \xi_2)+  I^{\alpha} (\xi_2^{-1} \otimes \nu^{\pm 1})
       + I^{\alpha} (\nu^{\mp 1} \xi_2 \otimes \xi_2^{-1}) \\ & +I^{\alpha} (\nu^{\mp 1} \xi_2 \otimes \nu^{\pm 1}) + I^{\alpha} (\nu^{\mp 1} \otimes \nu^{\pm 1} \xi^{-1}_2)  + I^{\alpha} (\xi_2 \otimes \nu^{\pm 1} \xi_2^{-1}) .
    \end{aligned}
  \end{equation}

\end{lem}
\begin{proof}
  We know that $r_{\alpha}(I(\nu^{\mp 1} \otimes \xi_2)) = r_{\alpha} (I_{\alpha}
  (I^{\alpha}(\nu^{\mp 1} \otimes \xi_2 )))$. Applying
  \cite{BernsteinZelevinsky1977}  to the functor $r_{\alpha} \circ
I_{\alpha}$, we have:
  \begin{equation} \label{eq:raInuFmp1xi2}
    \begin{aligned}
      r_{\alpha}I(\nu^{\mp 1} \otimes \xi_2) &= I^{\alpha} (\nu^{\mp 1} \otimes \xi_2)+ s_{3 \alpha + 2 \beta} \circ I^{\alpha} (\nu^{\mp 1} \otimes \xi_2) + I^{\alpha} \circ s_{\beta} \circ r^{L_{\alpha}}_T (I^{\alpha} (\nu^{\mp 1} \otimes \xi_2)) \\
      &+ I^{\alpha} \circ s_{\alpha + \beta} \circ r^{L_{\alpha}}_T (I^{\alpha} (\nu^{\mp 1} \otimes \xi_2)).
    \end{aligned}
  \end{equation}
  The same theorem applied to $r^{L_{\alpha}}_T  \circ I^{\alpha} $ above, we
  obtain that
  \begin{equation}
    r^{L_{\alpha}}_T (I^{\alpha} (\nu^{\mp 1} \otimes \xi_2)) = \nu^{\mp 1} \otimes \xi_2 + \xi_2 \otimes \nu^{\mp 1}.
  \end{equation}
  Thus the following holds
  \begin{equation} \label{eq:raInuFmp1xi2}
    \begin{aligned}
      r_{\alpha}I(\nu^{\mp 1} \otimes \xi_2) &= I^{\alpha} (\nu^{\mp 1} \otimes \xi_2)+  I^{\alpha} (\xi_2^{-1} \otimes \nu^{\pm 1})
       + I^{\alpha} (\nu^{\mp 1} \xi_2 \otimes \xi_2^{-1}) \\ & +I^{\alpha} (\nu^{\mp 1} \xi_2 \otimes \nu^{\pm 1}) + I^{\alpha} (\nu^{\mp 1} \otimes \nu^{\pm 1} \xi^{-1}_2)  + I^{\alpha} (\xi_2 \otimes \nu^{\pm 1} \xi_2^{-1}) .
    \end{aligned}
  \end{equation}
\end{proof}

We deduce from the above Lemma that
 \begin{equation} \label{eq:r0InuFmp1xi2}
    \begin{aligned}
      r_{\emptyset}I(\nu^{\mp 1} \otimes \xi_2) &= \nu ^{\mp 1} \otimes \xi_2+  \xi_2 \otimes \nu^{\mp 1} +  \xi_2^{-1} \otimes \nu^{\pm 1} +  \nu^{\pm 1} \otimes \xi_2^{-1} + \nu^{\mp 1} \xi_2 \otimes \xi_2^{-1}\\ & + \xi_2^{-1} \otimes \nu^{\mp 1} \xi_2 + \nu^{\mp 1} \xi_2 \otimes \nu^{\pm 1} + \nu^{\pm 1} \otimes \nu^{\mp 1}  \xi_2 + \nu^{\mp 1} \otimes \nu^{\pm 1} \xi_2^{-1}  \\ & + \nu^{\pm 1} \xi_2^{-1} \otimes \nu^{\mp 1} + \xi_2 \otimes \nu^{\pm 1} \xi_2^{-1} + \nu^{\pm 1} \xi_2^{-1} \otimes \xi_2 .
    \end{aligned}
  \end{equation}
By \cite{Rodier1981}, we know $I(\nu^{\mp 1}
\otimes \xi_2)$ has length $2 \times \operatorname{Card}(\{ s_{3\alpha + 2 \beta}  \}) =2$. Let $\pi (\nu^{\mp 1} \otimes \xi_2)$
denote its irreducible subrepresentation and $J(\nu^{\mp 1} \otimes \xi_2)$
denote its irreducible quotient.

If $\xi_2$ is ramified cubic, $\mathfrak{J}^{\mathfrak{s}} = \SL_3 (\C)$ and
$\H^{\mathfrak{s}} = \H ( \operatorname{PGL}_3 (F), 1)$.  \cite[Section 9.3.1 case (a)]{AubertXu2023explicit}
and \cite[Table 4.1]{Ram2004Rank2}, $\pi (\nu^{\mp 1} \otimes \xi_2)$ and  $J(\nu^{\mp 1} \otimes \xi_2)$
correspond to the standard module indexed by $(t_b, e_{\alpha_2},1)$ and $(t_b,
0, 1)$ respectively. We conclude that
\[
  D_{\H^{\mathfrak{s}}}([M_{t_b, e_{\alpha_2},1} ]) = [M_{t_b, e_{\alpha_2},1}^{*}] = [M_{t_b,
    0, 1}].
\]

We can see easily that if we require $I^{\alpha}(\nu^{\mp 1} \otimes \xi_2 \nu)$ to
be irreducible, then $\xi_2$ is not quadratic. If $\xi_2$ is unramified, then  $I^{\alpha}(\nu^{\mp 1} \otimes \xi_2 \nu)$
being irreducible forces it not to be of any finite order.

If $\xi_2$ is ramified non-cubic, $\mathfrak{J}^{\mathfrak{s}} = \GL_2 (\C)$ and
$\H^{\mathfrak{s}} = \H ( \operatorname{GL}_2 (F), 1)$. By
\cite[Section 9.3.1 case (b)]{AubertXu2023explicit}
and \cite[Table 2.1]{Ram2004Rank2}, $\pi (\nu^{\mp 1} \otimes \xi_2)$ and  $J(\nu^{\mp 1} \otimes \xi_2)$
correspond to the standard module indexed by $(t_a, e_{\alpha_1},1)$ and $(t_a,
0, 1)$ respectively. We conclude that
\[
  D_{\H^{\mathfrak{s}}}([M_{t_a, e_{\alpha_1},1} ]) = [M_{t_a, e_{\alpha_1},1}^{*}] = [M_{t_a,
    0, 1}].
\]

If $\xi_2$ is unramified and $s_2 \neq \pm 1$ (required by $I^{\beta}
(\nu^{\mp 1} \otimes \xi_2 \nu)$ being irreducible), then
$\mathfrak{J}^{\mathfrak{s}}= G_2 (\C)$ and
$\H^{\mathfrak{s}} = \H ( G_2 (F), 1)$.  By \cite[Section
9.3.1 case (c)]{AubertXu2023explicit} and \cite[Table 6.1]{Ram2004Rank2},  $\pi (\nu^{\mp 1}
\otimes \xi_2)$ and  $J(\nu^{\mp 1} \otimes \xi_2)$
correspond to the standard module indexed by $(t_g, e_{\alpha_1},1)$ and $(t_g,
0, 1)$ respectively. We conclude that
\[
  D_{\H^{\mathfrak{s}}}([M_{t_g, e_{\alpha_1},1} ]) = [M_{t_g, e_{\alpha_1},1}^{*}] = [M_{t_g,
    0, 1}].
\]

\subsection{Two lemmas for computation}
Before we discuss the case when $I^{\alpha}(\nu^{\mp 1} \otimes \xi_2 \nu)$
is reducible, we need to introduce some lemmas for computational convenience.

\begin{lem}\label{lem:ra}
  Let $\chi \in \operatorname{Irr}(F^{\times})$. We have in $\operatorname{R}(L_{\alpha})$:
\begin{equation*}
    \begin{aligned}
    & r_{\alpha} \left( I_{\alpha} \left( s, \delta (\chi) \right) \right)  = \nu^s \delta (\chi) +\nu^{-s} \delta (\chi^{-1}) + I^{\alpha} (\nu^{2s} \chi^2 \otimes \nu^{-s+1/2} \chi^{-1})+I^{\alpha} (\nu^{s+1/2 } \chi \otimes \nu^{-2s} \chi^{-2}). \\
   & r_{\alpha} \left( I_{\alpha}(s, \chi \circ \det) \right) = \nu^s  \chi \circ \det +\nu^{-s}\chi^{-1} \circ \det + I^{\alpha} (\nu^{s-1/2} \chi \otimes \nu^{-2s} \chi^{-2})+I^{\alpha} (\nu^{2s} \chi^2  \otimes \nu^{-s-1/2} \chi^{-1}).\\
  &  r_{\alpha} \left( I_{\beta} \left( s, \delta (\chi) \right) \right) = I^{\alpha} (\nu^{s-1/2} \chi \otimes \nu) + I^{\alpha} (\nu \otimes \nu^{-(s+1/2)} \chi^{-1}) + I^{\alpha} (\nu^{s+1/2} \chi \otimes \nu^{-s+1/2} \chi^{-1}).\\
 & r_{\alpha} \left( I_{\beta}\left( s, \chi \circ \det \right) \right) = I^{\alpha} (\nu^{s+1/2} \chi \otimes \nu^{-1}) + I^{\alpha} (\nu^{-1} \otimes \nu^{-s+1/2} \chi^{-1}) + I^{\alpha} (\nu^{s-1/2} \chi \otimes \nu^{-s-1/2} \chi^{-1}).
  \end{aligned}
\end{equation*}
If we use \cite[Lemma 5.4 (iii)]{BernsteinDeligneKazhdan1986}, we also have:
\begin{equation*}
    r_{\alpha} \left( I_{\beta}(s, \chi \circ \det) \right) = I^{\alpha} (\nu^{-s+1/2} \chi^{-1} \otimes \nu^{-1}) + I^{\alpha} (\nu^{-1} \otimes \nu^{s+1/2} \chi) + I^{\alpha} (\nu^{-s-1/2} \chi^{-1} \otimes \nu^{s-1/2} \chi).
  \end{equation*}
\end{lem}

\begin{lem}\label{lem:rb}
 Let $\chi \in \operatorname{Irr}(F^{\times})$. In
 $\operatorname{R}(L_{\beta})$, we have:
 \begin{equation*}
   \begin{aligned}
    & r_{\beta} \left(I_{\beta}\left(s, \delta (\chi) \right) \right) = \nu^s \delta (\chi) + \nu^{-s} \delta (\chi^{-1}) + I^{\beta} (\nu \otimes \nu^{s-1/2} \chi) + I^{\beta} (\nu^{-s+1/2} \chi^{-1} \otimes \nu^{s+1/2} \chi).\\
   & r_{\beta} \left( I_{\beta}(s, \chi \circ \det) \right) = \nu^s \chi \circ \det + \nu^{-s}  \chi^{-1} \circ \det + I^{\beta} (\nu^{-1} \otimes \nu^{s+1/2} \chi) + I^{\beta} (\nu^{-s-1/2} \chi^{-1} \otimes \nu^{s-1/2} \chi).\\
  &  r_{\beta} \left( I_{\alpha}(s, \delta (\chi)) \right) = I^{\beta} (\nu^{s+1/2} \chi \otimes \nu^{s-1/2} \chi) + I^{\beta} (\nu^{-2s} \chi^{-2} \otimes \nu^{s+1/2} \chi) + I^{\beta} (\nu^{-s+1/2} \chi^{-1} \otimes \nu^{2s} \chi^2).\\
   & r_{\beta} \left( I_{\alpha}(s, \chi \circ \det) \right) = I^{\beta} (\nu^{s-1/2} \chi \otimes \nu^{s+1/2} \chi) +I^{\beta}(\nu^{-2s} \chi^{-2} \otimes \nu^{s-1/2} \chi) + I^{\beta} (\nu^{-s-1/2} \chi^{-1} \otimes \nu^{2s} \chi^2).
    \end{aligned}
\end{equation*}

\end{lem}

In the following sections, we will determine how the duality operator acts when
$I_{\gamma} (s, \delta (\chi))$ and $I_{\gamma} (s, \chi \circ \det)$ ($\gamma =
\alpha$ or $\beta$) reduce. We will discuss based on the value of $s$ and $\chi$ in $I_{\alpha} (s, \delta (\chi))$ and $I_{\alpha} (s, \chi \circ \det)$. Excluding the cases listed in  \cite[Theorem 3.1]{Muic1997}, \cite[Lemma 3.1]{Muic1997} tells us that $D_{G_2} (I_{\gamma} (s,
\delta (\chi))) = I_{\gamma} (s, \chi \circ \det)$ if $\chi$ is a unitary
character.

\subsection{Case by case discussion, continued}
\subsubsection{Case \refeq{Muic1997itm:2} within the case $\mathfrak{s} = [T, 1]_G$ Part I}
We start by discussing the case when $I^{\alpha}(\nu^{\mp 1 -s} \otimes \xi_2 \nu^{-s})$ is
  reducible. This condition is equivalent to
\[
  \nu^{\mp 1} \otimes \xi_2 \in \{ \nu \otimes 1, \ \nu \otimes
  \nu^{2}, \ \nu^{-1} \otimes 1 , \ \nu^{-1} \otimes \nu^{-2} \}.
  \]

We first discuss the case when $\nu^{\mp 1} \otimes \xi_2 \in \{ \nu \otimes 1, \ \nu^{-1}
  \otimes 1 \}$. This is equivalent to $s =1/2$ and $\chi = 1$ in Lemma \ref{lem:ra} and Lemma \ref{lem:rb}. This part is a special case of the cases when $s =1/2$ and $\chi^2 = 1$ or
$\chi^3=1$, but the arguments are slightly different because

\begin{enumerate}[(a)]
\item Some representations will be reducible, like $I^{\alpha} (\nu \otimes 1)$
  in \refeq{eq:raIaSt1}.
\item $\nu \otimes 1$ and $\nu \otimes \nu^{-1}$ are not regular \emph{i.e.}
  $s_{\alpha}s_{\beta}s_{\alpha}  \left( \nu \otimes 1\right) =\nu \otimes 1$, we can not do the length estimate.
\item We have less isomorphic representations.
\end{enumerate}

If $s=1/2$ and $\chi =1$, then by \cite[Theorem 3.1]{Muic1997} the four terms $I_{\gamma} (s, \delta (\chi))$
and $I_{\gamma} (s, \chi \circ \det)$, $\gamma = \alpha$, $\beta$ all reduces.

If $s=1/2$ and $\chi=1$, then $I( 1 \otimes \nu^{-1}) = I (\nu ^{-1} \otimes
\nu) = I(1 \otimes \nu) \ [\operatorname{R}(G_2)] $ (we omit the terms differ by
$s_{\alpha}$). And as representations, we have one isomorphism only $I (\nu \otimes \nu^{-1}) \cong
I(\nu^{-1} \otimes \nu)$.

We have
\begin{equation}
  \label{eq:IaIa=IbIb}
   I_{\beta} (\nu^{1/2} \delta (1))+I_{\beta} (\nu^{1/2}  \circ \det)
  (=I (1 \otimes \nu))=I_{\alpha} (\nu^{1/2}  \circ \det) + I_{\alpha}
  (\nu^{1/2} \delta (1))   \ [\operatorname{R}(G_2)],
\end{equation}
where $I_{\beta} (\nu^{1/2}
  \delta (1))$ and $I_{\alpha} (\nu^{1/2}  \circ \det) $ are both
  subrepresentations of $I(1 \otimes \nu)$.

\begin{cor}\label{cor:ra3}
  If we take $s=1/2$, $\chi =1$ in Lemma \ref{lem:ra}, we
  have in $\operatorname{R}(L_{\alpha})$:
  \begin{equation}\label{eq:raIaSt3}
    \begin{aligned}
    r_{\alpha} (I_{\alpha}(1/2, \delta (1))) &= \nu^{1/2} \delta (1) +\nu^{-1/2} \delta (1) + I^{\alpha} (\nu  \otimes  1)+I^{\alpha} (\nu \otimes \nu^{-1} ) \\
    & = 2 \nu^{1/2} \delta (1) +\nu^{-1/2} \delta (1)+\nu^{1/2} \circ \det +I^{\alpha} (\nu \otimes \nu^{-1} ).
    \end{aligned}
\end{equation}

\begin{equation}\label{eq:raIaTriv3}
  \begin{aligned}
    r_{\alpha} (I_{\alpha}(1/2,   1_{\GL_2})) &= \nu^{1/2}   \circ \det +\nu^{-1/2} \circ \det + I^{\alpha} ( 1 \otimes \nu^{-1} )+I^{\alpha} (\nu   \otimes \nu^{-1} )\\
    &= \nu^{1/2}   \circ \det +2 \nu^{-1/2} \circ \det + \nu^{-1/2} \delta (1) +I^{\alpha} (\nu   \otimes \nu^{-1} ).
     \end{aligned}
\end{equation}

\begin{equation}\label{eq:raIbSt3}
   \begin{aligned}
     r_{\alpha} (I_{\beta}(1/2, \delta (1))) &= I^{\alpha} ( 1 \otimes \nu) + I^{\alpha} (\nu \otimes \nu^{-1} ) + I^{\alpha} (\nu  \otimes  1) \\
     & =2\nu^{1/2} \delta (1) + 2\nu^{1/2}  \circ  \det  + I^{\alpha} (\nu \otimes \nu^{-1} ).
      \end{aligned}
\end{equation}
\begin{equation}\label{eq:raIbTriv3}
   \begin{aligned}
     r_{\alpha} (I_{\beta}(1/2,  1_{\GL_2})) & = I^{\alpha} (\nu  \otimes \nu^{-1}) + I^{\alpha} (\nu^{-1} \otimes  1) + I^{\alpha} ( 1 \otimes \nu^{-1} ) \\
     & = 2\nu^{-1/2} \delta (1) +2 \nu^{-1/2}  \circ \det+ I^{\alpha} (\nu  \otimes \nu^{-1}) .
     \end{aligned}
\end{equation}

\end{cor}

\begin{cor}\label{cor:rb3}
  If we take $s=1/2$, $\chi =1$ in lemma \ref{lem:rb}, we
  have in $\operatorname{R}(L_{\beta})$:
  \begin{equation}\label{eq:rbIbSt3}
    \begin{aligned}
      r_{\beta} (I_{\beta}(1/2, \delta (1))) & = \nu^{1/2} \delta (1) + \nu^{-1/2} \delta (1) + I^{\beta} (\nu \otimes  1) + I^{\beta} ( 1 \otimes \nu )\\
      & = 2 \nu^{1/2} \delta (1) + \nu^{-1/2} \delta (1) +\nu^{1/2} \circ \det+ I^{\beta} (\nu \otimes  1) .
\end{aligned}
\end{equation}

  \begin{equation}\label{eq:rbIbTriv3}
    \begin{aligned}
      r_{\beta} (I_{\beta}(1/2,    1_{\GL_2}))  & = \nu^{1/2}  \circ \det + \nu^{-1/2}   \circ \det + I^{\beta} (\nu^{-1} \otimes \nu ) + I^{\beta} (\nu^{-1}  \otimes  1) \\
      &=  \nu^{1/2}  \circ \det +2 \nu^{-1/2}   \circ \det + \nu^{-1/2} \delta (1) + I^{\beta} (\nu^{-1}  \otimes  1).
      \end{aligned}
\end{equation}

\begin{equation}\label{eq:rbIaSt3}
    \begin{aligned}
      r_{\beta} (I_{\alpha}(1/2, \delta (1))) &= I^{\beta} (\nu  \otimes  1) + I^{\beta} (\nu^{-1}  \otimes \nu ) + I^{\beta} ( 1 \otimes \nu )\\
      &= \nu^{1/2} \delta (1) + \nu^{1/2}  \circ \det+\nu^{-1/2} \delta (1) + \nu^{-1/2} \circ \det + I^{\beta} (\nu  \otimes  1) .
    \end{aligned}
\end{equation}

 \begin{equation}\label{eq:rbIaTriv3}
   \begin{aligned}
     r_{\beta} \left(I_{\alpha}(1/2,  1_{\GL_2}) \right)& = I^{\beta} ( 1 \otimes \nu ) +I^{\beta}(\nu^{-1} \otimes 1) + I^{\beta} (\nu^{-1}  \otimes \nu)\\
      &= \nu^{-1/2} \delta (1) + \nu^{-1/2}  \circ \det+  \nu^{1/2} \delta (1) + \nu^{1/2} \circ \det +I^{\beta}(\nu^{-1} \otimes 1) .
     \end{aligned}
 \end{equation}

In \refeq{eq:rbIaTriv3} and \refeq{eq:rbIaSt3}, we used \cite[Lemma 5.4
(iii)]{BernsteinDeligneKazhdan1986}:
\[I^{\beta}(\chi_1 \otimes \chi_2) = I^{\beta}
  \left(s_{\beta}(\chi_1 \otimes \chi_2) \right).
\]
in $\operatorname{R}(L_{\beta})$.
\end{cor}

The following Proposition is from \cite{Muic1997}, but our proof is a bit
different which enables us to know more about the images of the composition
factors under Jacquet functors.
\begin{prop}[G. Muic Proposition 4.3]\label{prop:I1v}
  Suppose that $\chi =1$, $s=1/2$ then we have the following.
  \begin{enumerate}[(1)]
  \item The induced representation $I(1 \otimes \nu )$ contains exactly two
    irreducible subrepresentations $\pi(1)$ and $\pi'(1)$. We have $r^{G_2}_T (\pi
    (1)) = 1 \otimes \nu$, $r^{G_2}_T (\pi'(1)) = 1 \otimes \nu + 2 (\nu \otimes
    1)$.
  \item In $\operatorname{R}(G_2)$, we have
    \begin{equation}\label{eq:IaIbAll3}
           \begin{aligned}
             I_{\alpha} (1/2, \delta (1)) & = \pi'(1) +J_{\alpha} (1/2, \delta (1)) +J_{\beta} (1/2, \delta (1)), \\
             I_{\beta} (1/2, \delta (1)) & = \pi (1) + \pi'(1) +J_{\beta} (1/2, \delta (1)), \\
             I_{\alpha} (1/2, 1_{\GL_2}) &= \pi (1) +J_{\beta} (1, \pi (1,1))+J_{\beta} (1/2, \delta (1)), \\
             I_{\beta} (1/2,  1_{\GL_2}) & = J_{\beta} (1, \pi (1,1)) +J_{\beta} (1/2, \delta (1)) +J_{\alpha} (1/2, \delta (1)).
           \end{aligned}
         \end{equation}
  \end{enumerate}
\end{prop}

\begin{proof}

  By Langlands quotient theorem applied to the following: $I(\nu \otimes
  1) = I_{\beta} (1,\pi (1,1))$, we
  see $J_{\beta} (1, \pi (1, 1))$ is the unique irreducible quotient of $I( \nu
  \otimes 1)$. Consider the exact sequence:

  \[
0 \rightarrow I_{\alpha} (1/2, \delta (1)) \rightarrow I (\nu \otimes 1) \rightarrow
I_{\alpha} (1/2, 1_{\GL_2}) \rightarrow 0 .
  \]
We know $J_{\beta} (1, \pi (1, 1))$ is the unique irreducible quotient of
$I_{\alpha} (1/2, 1_{\GL_2})$.

  Then Langlands quotient theorem applied to $I_{\alpha} (1/2, \delta (1))$
  and $I_{\beta} (1/2, \delta (1))$, we find two more composition factors: $J_{\alpha} (1/2, \delta (1))$ and $J_{\beta} (1/2,
  \delta (1))$ as the unique irreducible quotient of  $I_{\alpha} \left(1/2,
    \delta (1) \right)$
  and $I_{\beta} \left(1/2, \delta (1) \right)$ respectively.

  We have
  \begin{equation}
    \label{eq:JbIv-1vinclus}
    J_{\beta} \left(1/2, \delta (1) \right) \hookrightarrow I_{\beta} \left(-1/2, \delta (1) \right) \hookrightarrow I(\nu^{-1} \otimes \nu) \cong I(\nu \otimes \nu^{-1}).
  \end{equation}

  From
  \begin{equation*}
    \begin{aligned}
    0 \neq & \Hom_{G_2}\left( J_{\beta} \left(1/2, \delta (1) \right),  I(\nu^{-1} \otimes
  \nu) \right) = \Hom_{L_{\alpha}} \left( r_{\alpha} \left(J_{\beta} (1/2, \delta (1)) \right),
  I^{\alpha}(\nu^{-1} \otimes \nu) \right) ,\\ 0 \neq & \Hom_{G_2} \left( J_{\beta} \left(1/2, \delta (1) \right), I_{\beta} \left( -1/2, \delta (1) \right) \right)=   \Hom_{L_{\beta}} \left( r_{\beta} \left( J_{\beta} \left(1/2, \delta (1) \right) \right),
  \nu^{-1/2} \delta (1)) \right), \\
  0 \neq & \Hom_{G_2}\left( J_{\beta} (1/2, \delta (1)),I(\nu^{-1} \otimes
  \nu) \right)=\Hom_{L_{\beta}} \left( r_{\beta}\left( J_{\beta} \left(1/2, \delta (1) \right) \right),
  I^{\beta}(\nu \otimes \nu^{-1}) \right).
    \end{aligned}
  \end{equation*}
We conclude that $r_{\alpha}(J_{\beta} (1/2, \delta (1)))$ contains
  $I^{\alpha}(\nu^{-1} \otimes \nu)$, $r_{\beta}(J_{\beta} (1/2, \delta (1)))$
  contains $ \nu^{-1/2} \delta (1)$ and  $r_{\beta}(J_{\beta} (1/2, \delta
  (1)))$ also contains $\nu^{1/2}
  \delta (1)$ or $\nu^{1/2} \circ \det$, or both.

We know that  $I_{\beta} (1/2, 1_{\GL_2})$ is a subrepresentation of $I( \nu \otimes
  \nu^{-1})$ and we must have
  \begin{equation}
    \label{eq:Muic4.21}
   I_{\beta}\left( 1/2, \delta (1) \right) \left( =I_{\beta} \left( -1/2, \delta (1) \right) [ \operatorname{R}(G_2)] \right)  \bigcap_{[\operatorname{R}(G_2)]} I_{\beta} (1/2, 1_{\GL_2}) \neq 0 ,
  \end{equation}
 (where the intersection $ \bigcap_{[\operatorname{R}(G_2)]}$ in the above means we
 take their common factors in $[\operatorname{R}(G_2)]$), otherwise
  \[
J_{\beta} (1/2, \delta (1)) \hookrightarrow I_{\beta} (-1/2, \delta (1))
\hookrightarrow I(\nu \otimes \nu^{-1}) / I_{\beta} (1/2, 1_{\GL_2}) \cong
I_{\beta} (1/2, \delta (1)).
\]
This implies that $J_{\beta} (1/2, \delta (1))$ is a subrepresentation of
$I_{\beta} \left( 1/2, \delta (1) \right)$, then $I_{\beta} (1/2, \delta (1))$ is irreducible.
This contradicts \cite[Theorem 3.1]{Muic1997}.

From \refeq{eq:Muic4.21}, we compare \refeq{eq:raIbSt3} and \refeq{eq:raIbTriv3}
to find out the image of common factor under $r_{\alpha}$ is $I^{\alpha} (\nu
\otimes \nu^{-1})$, thus the only common factor of $I_{\beta}(1/2, \delta (1)) \bigcap_{[\operatorname{R}(G_2)]}
I_{\beta} (1/2, 1_{\GL_2})$   has to be  $J_{\beta} (1/2, \delta
(1))$, and $r_{\alpha} (J_{\beta} (1/2, \delta (1))) = I^{\alpha} (\nu
\otimes \nu^{-1}) $, $r_{\beta} (J_{\beta} (1/2, \delta (1))) = \nu^{-1/2}
\delta (1) + \nu^{1/2} \circ \det \ [\operatorname{R}(L_{\beta})]$. We find from \refeq{eq:raIaSt3} and
\refeq{eq:raIaTriv3} that $ I_{\alpha} (1/2, 1_{\GL_2})$
and $I_{\alpha}
  (1/2, \delta (1))$ both contain $J_{\beta} (1/2, \delta (1))$ as a
  composition factor.

Recall that $I_{\beta} (1/2, \delta (1))$, $I_{\alpha} (1/2,  1_{\GL_2}) $ are both
  subrepresentations of $I(1 \otimes \nu)$. We apply $-2J_{\beta} (1/2, \delta
  (1))$ to both sides of \refeq{eq:IaIa=IbIb}, and get
   \begin{equation}\label{eq:IaIa=IbIb2}
     \begin{aligned}
   & I_{\beta} (1/2, \delta (1))-J_{\beta} (1/2, \delta (1))+I_{\beta} (1/2, 1_{\GL_2})-J_{\beta} (1/2, \delta (1))   \\
    = &   I_{\alpha} (1/2, 1_{\GL_2})-J_{\beta} (1/2, \delta (1)) + I_{\alpha}
  (1/2, \delta (1)) -J_{\beta} (1/2, \delta (1))   \ [\operatorname{R}(G_2)],
     \end{aligned}
   \end{equation}
where $(I_{\beta} (1/2, \delta (1))-J_{\beta} (1/2, \delta (1)) ) \bigcap_{[\operatorname{R}(G_2)]} (I_{\beta}
(1/2, 1_{\GL_2})-J_{\beta} (1/2, \delta (1))) = \emptyset \ [\operatorname{R}(G_2)] $  by comparing the common terms in  \refeq{eq:raIbSt3} and \refeq{eq:raIbTriv3}.
  We have
  \[ \left( I_{\beta} (1/2, \delta (1)) \bigcap_{[\operatorname{R}(G_2)]} I_{\alpha} (1/2,
      1_{\GL_2}) \right) - J_{\beta} (1/2, \delta (1)) \neq 0.
  \]
  Otherwise
  \[
    I_{\alpha} (1/2, 1_{\GL_2})-J_{\beta} (1/2, \delta (1)) \subset
    I_{\beta} (1/2, 1_{\GL_2})-J_{\beta} (1/2, \delta (1)) \ [\operatorname{R}(G_2)]
  \]
  can be deduced from
  \refeq{eq:IaIa=IbIb2} whose  left hand side is a disjoint sum, but this will contradict  the fact that excluding the terms brought by
  $J_{\beta} (1/2, \delta (1))$, the remaining of \refeq{eq:raIaTriv3} is not included in the
  remaining of \refeq{eq:raIbTriv3}.
  Compare \refeq{eq:raIaTriv3} and
  \refeq{eq:raIbSt3} (for $r_{\beta}$, \refeq{eq:rbIaTriv3} and
  \refeq{eq:rbIbSt3} respectively), we know that besides $J_{\beta} (1/2, \delta (1))$, there is another irreducible
  subquotient contained in
\[
I_{\beta} (1/2, \delta (1)) \bigcap_{[\operatorname{R}(G_2)]} I_{\alpha} (1/2,
  1_{\GL_2})
\]
whose image under $r_{\alpha}$ is $\nu^{1/2} \det$ and under
  $r_{\beta}$ is $\nu^{1/2}\delta (1)$. If we look at any irreducible
  subrepresentation of $I_{\alpha} (\nu^{1/2}  \circ \det) $, by the adjointness
  of $r_{\alpha}$ and $I_{\alpha}$, we know the image of such irreducible
  subrepresentation under $r_{\alpha}$ must contain $\nu^{1/2} \det$. From the fact that $ r_{\alpha}(I_{\alpha} (1/2,
  1_{\GL_2}))$ contains $\nu^{1/2} \det$ with multiplicity one, we know the other
  composition factor of $I_{\beta} (1/2, \delta (1))  \bigcap_{[\operatorname{R}(G_2)]} I_{\alpha} (1/2,
  1_{\GL_2})$ is the unique irreducible subrepresentation of $ I_{\alpha} (1/2,
  1_{\GL_2})$, denoted by $\pi (1)$.

  Using the same argument to \refeq{eq:raIaSt3} and \refeq{eq:raIbTriv3} (for
  $r_{\beta}$, \refeq{eq:rbIbTriv3} and \refeq{eq:rbIaSt3}, respectively), we
  know that there is an irreducible subquotient of
  \[
    \left(I_{\beta} (1/2, 1_{\GL_2})
       \bigcap_{[\operatorname{R}(G_2)]} I_{\alpha} (1/2, \delta (1)) \right) - J_{\beta} (1/2, \delta (1))
    \]
whose image under $r_{\alpha}$ is $\nu^{-1/2} \delta (1)$, under $r_{\beta}$ is
$\nu^{-1/2} \circ \det$. But we have
\begin{equation}
  \label{eq:JaIa-1inclu}
  J_{\alpha} (1/2, \delta (1)) \hookrightarrow I_{\alpha} (-1/2, \delta (1)) \left(  = I_{\alpha} \left( 1/2, \delta (1) \right) [ \operatorname{R} (G_2)]\right)
\end{equation}
as the unique subrepresentation, $r_{\alpha}(J_{\alpha} (1/2, \delta (1))) $
contains $\nu^{-1/2} \delta (1)$. The multiplicity of $\nu^{-1/2} \delta (1)$ in
\refeq{eq:raIaSt3} is one, we thus conclude
\[
  I_{\beta} (1/2, 1_{\GL_2})
   \bigcap_{[\operatorname{R}(G_2)]} I_{\alpha} (1/2, \delta (1))  = J_{\beta} (1/2, \delta (1)) + J_{\alpha}
  (1/2, \delta (1)).
  \]

  From the inclusion
  \[
J_{\beta}(1, \pi (1,1)) \hookrightarrow I_{\alpha} (-1/2, 1_{\GL_2}) ,
  \]
we know $r_{\alpha} (J_{\beta} (1, \pi (1,1)))$ contains $\nu^{-1/2} \circ
\det$. Since $I_{\alpha} (-1/2, 1_{\GL_2}) \hookrightarrow I(\nu^{-1} \otimes 1)$,
we know $r_{\beta} (J_{\beta}(1, \pi (1,1)) ) $ contains $I^{\beta} (\nu^{-1}
\otimes 1)$.

We may apply similar argument to any irreducible subrepresentation of
$I_{\alpha} (1/2, \delta (1))$, and we find the image under $r_{\alpha}$ must
contain $\nu^{1/2} \delta (1)$, the image under $r_{\beta}$ of such
irreducible subrepresentation must contain $I^{\beta} (\nu \otimes 1)$ which is
of multiplicity $1$ in $r_{\beta} (I_{\alpha}(1/2, \delta (1)))$, thus there is
a unique irreducible subrepresentation of $I_{\alpha} (1/2, \delta (1))$. We
denote this representation by $\pi'(1)$.

Suppose that $r_{\beta} \pi'(1) =I^{\beta} (\nu \otimes 1) + m \nu^{1/2} \delta
(1)$ where $m=0$ or $1$. If $m =0$, then there may exist another
subquotient $\tau$ such that $r_{\beta} (\tau) = \nu^{1/2} \delta (1)$, if we
assume
\[
  r_{\alpha} (\pi'(1)) = \nu^{1/2} \delta (1)+ s\nu^{1/2} \delta (1) + t
  \nu^{1/2} \circ \det
\]
where $s, t = 0$ or $1$ but they do not equal to $1$ at the same time, then
\[
  r_{\alpha} (\tau) = (1-s)\nu^{1/2} \delta (1) + (1-t)
  \nu^{1/2} \circ \det .
\]
If $s=0$, $t=1$, then
\[
  0 \neq \Hom_{L_{\alpha}}
(r_{\alpha} (\tau), \nu^{1/2} \delta (1)) = \Hom_{G_2} (\tau, I_{\alpha} (1/2,
\delta (1)))
\]
would contradicts the uniqueness of $\pi'(1)$ as subrepresentation of
$I_{\alpha} (1/2, \delta (1))$. Similarly, if $s=1$, $t=0$, then it leads to a
contradiction to uniqueness of $\pi(1)$. If $s =t=0$, then $r_{\alpha} D_{G_2}
(\pi' (1))   $ is not included in $r_{\alpha} (I_{\alpha} (1/2, 1_{\GL_2}))$
contradicts \cite[Lemma 3.1]{Muic1997}.

We conclude that $m=1$, and
\[
  r_{\beta} \pi'(1) =I^{\beta} (\nu \otimes 1) + \nu^{1/2} \delta (1), \ r_{\alpha} \pi' (1) = 2 \nu^{1/2}
  \delta (1) + \nu^{1/2} \circ \det .
\]
By the same argument,
  \[
    r_{\beta} (J_{\beta} (1, \pi (1, 1))) = I^{\beta} (\nu^{-1} \otimes 1) +
    \nu^{-1/2} \circ \det,
  \]
  and
  \[
    r_{\alpha} (J_{\beta} (1, \pi (1, 1))) = 2
    \nu^{-1/2} \det + \nu ^{-1/2} \delta (1).
    \]
\end{proof}

We list the images of composition factors under $r_{\alpha}$ and $r_{\beta}$ in
the following proposition:

\begin{prop}\label{prop:rarbAll3}
  Based on the proof of Proposition \ref{prop:I1v}, we have the following
  \begin{equation}\label{eq:raraAll3}
  \begin{aligned}
    r_{\alpha} (\pi (1))= \nu^{1/2} \circ \det, & \ r_{\beta} (\pi (1)) = \nu^{1/2} \delta (1), \\
    r_{\alpha} (J_{\alpha} (1/2, \delta (\chi)))
    =\nu^{-1/2} \delta (1) , & \ r_{\beta} (J_{\alpha}(1/2, \delta (\chi))) = \nu^{-1/2} \circ \det, \\
    r_{\alpha} \left( J_{\beta} (1/2, \delta (1)) \right) = I^{\alpha} (\nu \otimes \nu^{-1}), & \  r_{\beta} \left( J_{\beta} (1/2, \delta (1)) \right) = \nu^{-1/2} \delta (1) + \nu^{1/2} \circ \det , \\
    r_{\alpha} (\pi'(1)) = 2 \nu^{1/2} \delta (1)+ \nu^{1/2} \circ \det, & \ r_{\beta} (\pi'(1)) =I^{\beta} (\nu \otimes 1) +  \nu^{1/2} \delta
(1) ,\\
    r_{\alpha} (J_{\beta} (1, \pi (1, 1)))  = 2 \nu^{-1/2} \circ \det + \nu^{-1/2} \delta (1), & \  r_{\beta} (J_{\beta} (1, \pi (1, 1))) = I^{\beta} (\nu^{-1} \otimes 1)
                                                                                                 + \nu^{-1/2} \circ \det.
  \end{aligned}
\end{equation}
\end{prop}

\begin{prop}\label{prop:Duality3}
  We compute the Aubert-Zelevinsky duality of all the irreducible
  representations listed above:
  \begin{equation}\label{eq:Duality3}
    \begin{aligned}
      D_{G_2} (\pi (1)) = J_{\alpha} (1/2, \delta (1)),& \quad D_{G_2} (J_{\beta} (1/2, \delta (1))) =J_{\beta} (1/2, \delta (1)) \\
      D_{G_2} ( \pi ' (1))&=  J_{\beta} (1, \pi (1,1)).
    \end{aligned}
  \end{equation}
\end{prop}

\begin{proof}
  We compute as follows
  \begin{equation}\label{eq:DGpi1}
  \begin{aligned}
    & D_{G_2} (\pi (1)) = I \circ r_T^{G_2} (\pi (1)) - I_{\alpha} \circ r_{\alpha} (\pi (1))  - I_{\beta} \circ r_{\beta} (\pi (1)) + \pi (1)\\
   &= I \circ (r^{L_{\alpha}}_T ( \nu^{1/2} \circ \det)) -I_{\alpha} ( \nu^{1/2} \circ \det) - I_{\beta } ( \nu^{1/2} \delta (1)) + \pi (1) \\
   & = I_{\alpha} (1/2, \delta (1))- I_{\beta} (1/2, \delta (1)) + \pi (1)\\
   & = J_{\alpha} (1/2, \delta (1)).
  \end{aligned}
\end{equation}
Similarly for the rest terms:
\begin{equation}\label{eq:DGJb12}
  \begin{aligned}
    & D_{G_2} (J_{\beta} (1/2, \delta (1))) = I \circ r_T^{G_2} (J_{\beta} (1/2, \delta (1))) - I_{\alpha} \circ r_{\alpha} (J_{\beta} (1/2, \delta (1)))  \\
   & - I_{\beta} \circ r_{\beta} (J_{\beta} (1/2, \delta (1))) + J_{\beta} (1/2, \delta (1)) \\
   &= I \circ (r^{L_{\alpha}}_T (I^{\alpha} (\nu \otimes \nu^{-1}))) -I_{\alpha} (I^{\alpha} (\nu \otimes \nu^{-1})) - I_{\beta } (\nu^{-1/2} \delta (1)) \\ &-I_{\beta}( \nu^{1/2} \circ \det )  + J_{\beta} (1/2, \delta (1)) \\
   & =J_{\beta} (1/2, \delta (1)) .
  \end{aligned}
\end{equation}
\begin{equation}\label{eq:DGpi'1}
  \begin{aligned}
    & D_{G_2} ( \pi ' (1) ) = I \circ r_T^{G_2} (\pi' (1)) - I_{\alpha} \circ r_{\alpha} (\pi' (1))
    - I_{\beta} \circ r_{\beta} (\pi '(1)) + \pi' (1)\\
    &= I \circ (r^{L_{\alpha}}_T (2\nu^{1/2} \delta (1) + \nu^{1/2} \circ \det)) -I_{\alpha} (2 \nu^{1/2} \delta (1)+ \nu^{1/2} \circ \det)  \\ &  - I_{\beta } (I^{\beta} (\nu \otimes 1) +  \nu^{1/2} \delta
(1)) + \pi' (1) \\
&= 2 I(\nu \otimes 1) -2 I_{\alpha} (1/2, \delta (1)) -I_{\alpha} (1/2, 1_{\GL_2}) -I_{\beta} (1/2, \delta (1)) + \pi' (1) \\
&=  I_{\alpha} (1/2, 1_{\GL_2}) -I_{\beta} (1/2, \delta (1)) + \pi'(1) \\
& = J_{\beta} (1, \pi (1,1)).
  \end{aligned}
\end{equation}

\end{proof}

In this case $\mathfrak{s} = [T, 1]$ and $\mathfrak{J}^{\mathfrak{s}} =
G_2(\C)$ and
$\H^{\mathfrak{s}} = \H ( G_2 (F), 1)$. By \cite[Section
9.3.1 table 15]{AubertXu2023explicit} and \cite[Table 6.1]{Ram2004Rank2}, we
have the following proposition.
\begin{prop}\label{prop:Duality3Hecke}
   Using
the indexing triples to denote the standard modules, we have:
  \begin{equation}\label{eq:Duality3Hecke}
    \begin{aligned}
      D_{\H^{\mathfrak{s}} }(M_{t_e, e_{\alpha^{\vee}} + e_{\alpha^{\vee}} + 2 \beta^{\vee}, (21)} )& =M_{t_e, e_{\alpha^{\vee}} + e_{\alpha^{\vee}} + 2 \beta^{\vee}, (21)}^{*}= M_{t_e, e_{\alpha^{\vee}}, 1},\\
      D_{\H^{\mathfrak{s}} } (M_{t_e, e_{\alpha^{\vee} + \beta^{\vee}}, 1}) &=M_{t_e, e_{\alpha^{\vee} + \beta^{\vee}}, 1}^{*}=M_{t_e, e_{\alpha^{\vee} + \beta^{\vee}}, 1} \\
      D_{\H^{\mathfrak{s}} } ( M_{t_e, e_{\alpha^{\vee}} + e_{\alpha^{\vee}} + 2 \beta^{\vee}, (3)} )&=M_{t_e, e_{\alpha^{\vee}} + e_{\alpha^{\vee}} + 2 \beta^{\vee}, (3)}^{*}= M_{t_e, 0, 1}.
    \end{aligned}
  \end{equation}
\end{prop}

In the next subsection, we continue to discuss the other possibility for $I^{\alpha}(\nu^{\mp 1 -s}
\otimes \xi_2 \nu^{-s})$ to be reducible.
\subsubsection{Case  \refeq{Muic1997itm:2} within the case $\mathfrak{s} = [T, 1]_G$ Part II}
 Now we deal with the case:  \(\nu^{\mp 1} \otimes \xi_2 \in \left\{ \nu \otimes \nu^2, \ \nu^{-1} \otimes \nu^{-2} \right\}\).
This reduces to the case $s=3/2$ and $\chi=1$ in Lemma \ref{lem:ra} and Lemma \ref{lem:rb}. If $s=3/2$ and $\chi=1$, then by \cite[Theorem 3.1]{Muic1997}, $I_{\alpha} (s,
\delta (\chi))$ and $I_{\alpha} (s, \chi \circ \det)$ reduce, while $I_{\beta} (s,
\delta (\chi))$ or $I_{\beta} (s, \chi \circ \det)$ do not.

If $s=3/2$ and $\chi=1$, then $I(\nu^2 \otimes \nu) = I(\nu^3 \otimes \nu^{-2})\
= I (\nu \otimes \nu^{-3}) =I (\nu^{-2} \otimes \nu^{-1}) = I(\nu^3 \otimes
\nu^{-1}) = I (\nu^2 \otimes \nu^{-3}) \   [\operatorname{R}(G_2)] $ (we omit the terms differ by $s_{\alpha}$). Therefore   $I_{\beta} (\nu^{5/2} \circ \det) +
  I_{\beta} (\nu^{5/2} \delta (1)) =I (\nu^2 \otimes \nu) =I_{\alpha} (\nu^{3/2} \circ \det) + I_{\alpha}
  (\nu^{3/2} \delta (1)) \ [\operatorname{R}(G_2)]$.

  \begin{cor}\label{cor:raAll4}
   We consider the case  when $s=3/2$ and $\chi=1$ in Lemma \ref{lem:ra}, list the equations below:
     \begin{equation}\label{eq:raIaSt4}
    \begin{aligned}
      r_{\alpha} (I_{\alpha}(3/2, \delta (1))) = \nu^{3/2} \delta (1) +\nu^{-3/2} \delta (1) + I^{\alpha} (\nu^{3}  \otimes \nu^{-1} )+I^{\alpha} (\nu^{2 } \otimes \nu^{-3} ),
      \end{aligned}
\end{equation}
\begin{equation}\label{eq:raIaTriv4}
  \begin{aligned}
    r_{\alpha} (I_{\alpha}(3/2, 1_{\GL_2})) = \nu^{3/2} \circ \det +\nu^{-3/2} \circ \det + I^{\alpha} (\nu \otimes \nu^{-3} )+I^{\alpha} (\nu^{3}  \otimes \nu^{-2} ),
    \end{aligned}
\end{equation}
\begin{equation}\label{eq:raIbSt4}
  \begin{aligned}
    r_{\alpha} (I_{\beta}(5/2, \delta (1))) = \nu^{3/2} \delta (1) + \nu^{3/2} \circ \det + I^{\alpha} (\nu \otimes \nu^{-3} ) + I^{\alpha} (\nu^{3}  \otimes \nu^{-2} ),
    \end{aligned}
\end{equation}
\begin{equation}\label{eq:raIbTriv4}
  \begin{aligned}
    r_{\alpha} (I_{\beta}(5/2, 1_{\GL_2})) =\nu^{-3/2} \delta (1) + \nu^{-3/2} \circ \det + I^{\alpha} (\nu^{3}  \otimes \nu^{-1})  + I^{\alpha} (\nu^{2}  \otimes \nu^{-3} ).
    \end{aligned}
  \end{equation}
  \end{cor}

  \begin{cor}\label{cor:rbAll4}
     We consider the case  when $s=3/2$ and $\chi=1$ in Lemma \ref{lem:rb}, list the equations below:
\begin{equation}\label{eq:rbIbSt}
   \begin{aligned}
     r_{\beta} (I_{\beta}(5/2, \delta (1))) = \nu^{5/2} \delta (1) + \nu^{-5/2} \delta (1) + I^{\beta} (\nu \otimes \nu^{2} ) + I^{\beta} (\nu^{-2}  \otimes \nu^{3} ),
     \end{aligned}
\end{equation}
\begin{equation}\label{eq:rbIbTriv}
  \begin{aligned}
    r_{\beta} (I_{\beta}(5/2, 1_{\GL_2})) = \nu^{5/2}  \circ \det + \nu^{-5/2}   \circ \det + I^{\beta} (\nu^{-1} \otimes \nu^{3} ) + I^{\beta} (\nu^{-3} \otimes \nu^2 ),
    \end{aligned}
\end{equation}
\begin{equation}\label{eq:rbIaSt}
  \begin{aligned}
    r_{\beta} (I_{\alpha}(3/2, \delta (1))) = \nu^{5/2} \delta (1) + \nu^{5/2} \circ \det + I^{\beta} (\nu^{-3}  \otimes \nu^{2} ) + I^{\beta} (\nu^{-1}  \otimes \nu^{3} ),
    \end{aligned}
\end{equation}
\begin{equation}\label{eq:rbIaTriv}
  \begin{aligned}
    r_{\beta} \left( I_{\alpha}(3/2, 1_{\GL_2}) \right) =  \nu^{-5/2} \delta (1) + \nu^{-5/2} \circ \det  +I^{\beta} (\nu  \otimes \nu^{2})  + I^{\beta} (\nu^{-2}  \otimes \nu^3 ).
    \end{aligned}
\end{equation}
  \end{cor}

  \begin{prop}\label{Muicprop4.4}
    \begin{enumerate}[(1)]
    \item  The induced representation $I(\nu^2 \otimes \nu)= \Ind_B^{G_2}
    (\delta_{G_2}^{1/2})$, it contains a unique irreducible subrepresentation
    $\St_{G_2}$ (the \emph{Steinberg representation} of $G_2$) and a unique irreducible
    quotient $1_{G_2}$ (the \emph{trivial representation}). Moreover, we have $r^{G_2}_T (\St) = \nu^2
    \otimes \nu$.
  \item We have the following equations in $R[G_2]$:
    \begin{equation}\label{eq:Muicprop4.4}
      \begin{aligned}
        I_{\alpha} (3/2, \delta (1))& = \St_{G_2} +J_{\alpha} (3/2, \delta (1))\\
        I_{\beta} (5/2, \delta (1)) & = \St_{G_2} +J_{\beta} (5/2, \delta (1)) \\
        I_{\alpha} (3/2, 1_{\GL_2}) & = 1_{G_2} +J_{\beta}(5/2, \delta (1)) \\
        I_{\beta} (5/2, 1_{\GL_2}) &= 1_{G_2} +J_{\alpha} (3/2, \delta (1))
      \end{aligned}
    \end{equation}
    \end{enumerate}

  \end{prop}

  \begin{proof}
    By a theorem of F. Rodier \cite{Rodier1981}, we know $I(\nu^2   \otimes \nu)$ has length $4$ and
    contains a unique subrepresentation, more precisely $I(\nu^2 \otimes \nu)=
    \Ind_B^{G_2}   (\delta_{G_2}^{1/2})$, the unique subrepresentation is
    $\St_{G_2}$ and it has a unique irreducible quotient $1_{G_2}$.

    By Langlands quotient theorem applied to  $I_{\alpha} (3/2, \delta (\chi))$
  and $I_{\beta} (5/2, \delta (\chi))$, we find the rest two (it has length $4$)
  composition factors: $J_{\alpha} (3/2, \delta (1))$ and $J_{\beta} (5/2,
  \delta (1))$.

Since $I_{\alpha} (3/2, \delta (1))$ and $I_{\beta} (5/2, \delta (1))$ are
subrepresentations of $I(\nu^2 \otimes \nu)$, they must contain $\St_{G_2}$. And $I_{\alpha} (3/2, 1_{\GL_2})$ and $I_{\beta} (5/2, 1_{\GL_2})$ are
quotients of $I(\nu^2 \otimes \nu)$, thus they both contain $1_{G_2}$. Part
$(2)$ of the proposition is proved.

 We know that $r_{\alpha} (\St_{G_2})$ must be contained in the common parts of
  \refeq{eq:raIaSt4} and \refeq{eq:raIbSt4}. Since $0 \neq \Hom_{G_2} (\St_{G_2},I_{\alpha} (\nu^{3/2} \delta (1)) )=
  \Hom_{L_{\alpha}}(r_{\alpha} (\St_{G_2}), \nu^{3/2} \delta (1) )$,
  $r_{\alpha} (\St_{G_2})$ must contain $\nu^{3/2} \delta (1)$. We easily have
  $r_{\alpha} (\St_{G_2}) = \nu^{3/2} \delta (1)$. Composing with
  $r^{L_{\alpha}}_{T}$, we prove part $(1)$ of the proposition.
  \end{proof}

  \begin{prop}\label{prop:rarbAll4}
    We list the images under $r_{\alpha}$ and $r_{\beta}$ of each composition
factors:
  \begin{equation}\label{eq:raraAll4}
  \begin{aligned}
    r_{\alpha} (\St_{G_2})&= \nu^{3/2} \delta(1), \quad r_{\beta} (\St_{G_2}) = \nu^{5/2} \delta (1) ,\\
     r_{\alpha} (J_{\alpha} (3/2, \delta (1)))
     &=\nu^{-3/2} \delta (1) + I^{\alpha} (\nu^{3}  \otimes \nu^{-1} )+I^{\alpha} (\nu^{2 } \otimes \nu^{-3} ) , \\
     r_{\beta} (J_{\alpha}(3/2, \delta (1)))& =  \nu^{5/2} \circ \det + I^{\beta} (\nu^{-3}  \otimes \nu^{2} ) + I^{\beta} (\nu^{-1}  \otimes \nu^{3} ) , \\
    r_{\alpha}  (1_{G_2})&= \nu^{-3/2} \circ  \det, \quad   r_{\beta}  (1_{G_2})= \nu^{-5/2} \circ  \det, \\
 r_{\alpha} (J_{\beta} (5/2, \delta (1)))
     &=\nu^{3/2} \circ \det + I^{\alpha} (\nu \otimes \nu^{-3} ) + I^{\alpha} (\nu^{3}  \otimes \nu^{-2} ) , \\
     r_{\beta} (J_{\beta} (5/2, \delta (1)))& =  \nu^{-5/2} \delta (1) + I^{\beta} (\nu \otimes \nu^{2} ) + I^{\beta} (\nu^{-2}  \otimes \nu^{3} ).
  \end{aligned}
\end{equation}
\end{prop}

\begin{prop}\label{prop:Duality4}
  We compute the Aubert-Zelevinsky duality of all the irreducible
  representations listed above:
  \begin{equation}\label{eq:Duality4}
    \begin{aligned}
      D_{G_2} (\St_{G_2}) = 1_{G_2}, \quad D_{G_2} (J_{\alpha} (3/2, \delta (1))) =J_{\beta} (5/2, \delta (1))
    \end{aligned}
  \end{equation}
\end{prop}

\begin{proof}
  This is well known. But as a test of previous proposition, we do the
  calculation:
\begin{equation}\label{eq:DGpi1}
  \begin{aligned}
    & D_{G_2} (\St_{G_2}) = I \circ r_T^{G_2} (\St_{G_2}) - I_{\alpha} \circ r_{\alpha} (\St_{G_2})  - I_{\beta} \circ r_{\beta} (\St_{G_2}) + \St_{G_2}\\
   &= I(\nu^2 \otimes \nu) -I_{\alpha} ( \nu^{3/2} \delta(1))  - I_{\beta } (\nu^{5/2} \delta (1))  + \St_{G_2} \\
   & = I_{\alpha} (3/2, 1_{\GL_2})  - I_{\beta } (\nu^{5/2} \delta (1))  + \St_{G_2} \\
   &= 1_{G_2} .
  \end{aligned}
\end{equation}
  By \cite[Lemma 3.1]{Muic1997}, we have the rest of this proposition.
\end{proof}

In this case $\mathfrak{s} = [T, \xi \otimes \xi]$ and $\mathfrak{J}^{\mathfrak{s}} =
G_2(\C)$ and
$\H^{\mathfrak{s}} = \H ( G_2 (F), 1)$.  By \cite[Section
9.3.2 table 18]{AubertXu2023explicit} and \cite[Table 6.1]{Ram2004Rank2},  we get the proposition for
modules of Hecke algebra:

\begin{prop}\label{prop:Duality4Hecke}
   Using the indexing triples to denote the standard modules, we have:
  \begin{equation}\label{eq:Duality4Hecke}
    \begin{aligned}
      D_{\H^{\mathfrak{s}}} (M_{{t_a, e_{\alpha^{\vee}}+ e_{\beta^{\vee}} , 1 }}) =M_{{t_a, e_{\alpha^{\vee}}+ e_{\beta^{\vee}} , 1 }}^{*}= M_{t_a, 0, 1}, \quad D_{\H^{\mathfrak{s} }} (M_{t_a, e_{\alpha^{\vee}}, 1}) =M_{t_a, e_{\alpha^{\vee}}, 1}^{*}= M_{t_a, e_{\beta^{\vee}}, 1}.
    \end{aligned}
  \end{equation}
\end{prop}

We now begin to discuss case \refeq{Muic1997itm:3}.
\subsubsection{Case  \refeq{Muic1997itm:3} and  \refeq{Muic1997itm:5}}\label{subsection:3and5}

In case \refeq{Muic1997itm:3}, $\xi_1=\nu^{s_1} \xi=\nu^{s_1-s_2} \xi_2=\nu^{
  \pm 1} \xi_2$. By Proposition \ref{prop:Mgamma}, we have
$$
\begin{aligned}
  I\left(\xi_1 \otimes \xi_2\right) & =I_\alpha\left( \delta (\nu^{\pm 1/2} \xi_2)
  \right)+I_\alpha\left( \nu^{\pm 1/2} \xi_2 \circ \operatorname{det}\right) \\
  & =I_\alpha\left(s, \delta\left(\nu^{s_2-s \pm 1 / 2} \xi\right)\right)+I_\alpha\left(s, \nu^{s_2-s \pm 1 / 2} \xi \circ \operatorname{det}\right) .
\end{aligned}
$$

In case \refeq{Muic1997itm:5}, $\xi_1=\nu^{s_1} \xi=\xi_2^{-2} \nu^{ \pm 1}$. Thus
$$
I\left(\xi_1 \otimes \xi_2\right)=I\left(\nu^{ \pm 1} \xi_2^{-2} \otimes \xi_2\right) = I\left(\nu^{ \pm 1} \xi_2 \otimes \xi_2\right) .
$$
Case \refeq{Muic1997itm:5} is reduced to case \refeq{Muic1997itm:3} above.

By \cite[Theorem 3.1 (i)]{Muic1997}, if the character $\chi:=\nu^{s_2-s \pm
  1 / 2} \xi$ is unitary, then $I_\alpha\left(s, \delta\left(\chi \right)\right)$ and
$I_\alpha\left(s, \chi \circ \det \right)$ are irreducible unless
$s= \pm 1 / 2, \chi$ is quadratic, $s= \pm 3 / 2,
\chi = 1 $ or $s= \pm 1 / 2, \chi$ is cubic. ($s= \pm 3 / 2,
\chi = 1 $ is already discussed)

\subsubsection{Case \refeq{Muic1997itm:3}  within the case $\mathfrak{s}=[T, \xi \otimes 1]_G$ the
  length 2 case}

If $\xi_2$ is ramified, then by the discussion in
\cite[Section 9.3.3]{AubertXu2023explicit} $\xi$ is not quadratic nor cubic, we
have $\mathfrak{J}^{\mathfrak{s}} =\GL_2(\C)$ and $\H^{\mathfrak{s}} = \H
(\GL_2(F), 1)$. The two irreducible composition factors $I_\alpha\left( \delta (\nu^{\pm 1/2} \xi_2)
\right)$ and  $I_\alpha\left( \nu^{\pm 1/2} \xi_2 \circ \operatorname{det}\right)$
correspond to the standard module indexed by $(t_a, e_{\alpha},1)$ and $(t_a,
0, 1)$ respectively.  Using
the indexing triples to denote the standard modules, we conclude that
\[
  D_{\H^{\mathfrak{s}}}([M_{t_a, e_{\alpha},1} ]) = [M_{t_a, e_{\alpha},1}^{*}] = [M_{t_a,
    0, 1}].
\]
\subsubsection{Case \refeq{Muic1997itm:3}  within the case $\mathfrak{s}=[T,  1]_G$ the
  length 2 case}
If $\xi_2$ is unramified, then $\xi_1$ is also unramified and we are in the case
$\mathfrak{s} = [T, 1]$ and $\mathfrak{J}^{\mathfrak{s}} =G_2(\C)$ and $\H^{\mathfrak{s}} = \H
(G_2(F), 1)$.  By \cite[Section
9.3.3 table 21]{AubertXu2023explicit} and \cite[Table 6.1]{Ram2004Rank2}, $I_\alpha\left( \delta (\nu^{\pm 1/2} \xi_2)
\right)$ and  $I_\alpha\left( \nu^{\pm 1/2} \xi_2 \circ \operatorname{det}\right)$
correspond to the standard module indexed by $(t_g, e_{\alpha},1)$ and $(t_g,
0, 1)$ respectively. Using
the indexing triples to denote the standard modules, we have:
\[
  D_{\H^{\mathfrak{s}}}([M_{t_g, e_{\alpha},1} ]) = [M_{t_g, e_{\alpha},1}^{*}] = [M_{t_g,
    0, 1}].
\]

In the following, we will discuss the cases when
\[
  I_\alpha\left(s,
  \delta\left(\nu^{s_2-s \pm 1 / 2} \xi\right) \right) \textrm{ and } I_\alpha\left(s,
  \nu^{s_2-s \pm 1 / 2} \xi \circ \operatorname{det}\right)
\]
are reducible. From \cite[Theorem 3.1]{Muic1997}, it is enough to consider only $s >0$.

\subsection{Case \refeq{Muic1997itm:3} computations for $s=1/2$ and $\chi^2 =1$, $\chi \neq 1$}
We use notations from \ref{subsection:3and5}. If $s=1/2$ and $\chi^2 =1$, then by \cite[Theorem 3.1]{Muic1997} the four terms $I_{\gamma} (s, \delta (\chi))$
and $I_{\gamma} (s, \chi \circ \det)$, $\gamma = \alpha$, $\beta$ all reduces.

If $s=1/2$ and $\chi^2=1$, then
\[
  I(\nu \chi \otimes \chi) = I (\chi \otimes
\nu)  = I (\chi \nu \otimes  \nu^{-1})  = I (\chi \otimes  \nu^{-1} \chi) = I(\nu^{-1} \chi \otimes \nu) = I(\chi \otimes
\nu^{-1})
\]
in $[\operatorname{R}(G_2)] $ (we omit the terms differ by
$s_{\alpha}$). Therefore
\begin{equation}
  \begin{aligned}
 I_{\beta} (\nu^{1/2} \chi \circ \det) +
  I_{\beta} \left( \nu^{1/2} \delta (\chi) \right) & =I (\chi \otimes \nu) = I (\nu \otimes \chi)=I
  (\nu \chi \otimes \chi)  \\
  &   =I_{\alpha} (\nu^{1/2} \chi \circ \det) + I_{\alpha}
  (\nu^{1/2} \delta (\chi)) \ [\operatorname{R}(G_2)] .
  \end{aligned}
\end{equation}
Moreover, since $I^{\beta} (\chi \nu \otimes \chi) = I^{\beta} (\nu
\otimes \chi) \ [\operatorname{R}(L_{\beta})]$ and they are irreducible, they are isomorphic as representations,
and
\[
  I(\chi \nu \otimes \chi) = I_{\beta}(I^{\beta} (\chi \nu \otimes \chi)) =I_{\beta}(I^{\beta} ( \nu \otimes \chi)) =
  I(\nu \otimes \chi)
\]
as representations. Similarly, we have
\[
  I(\nu \chi \otimes
  \chi) \cong I(\chi \otimes \nu) \cong I(\nu \otimes \chi)
\]
as representations and they are all isomorphic representations of $I (\nu \chi \otimes \chi)$.

\begin{cor}\label{cor:ra1}
  If we take $s=1/2$, $\chi^2 =1$ in Lemma \ref{lem:ra}, we
  have in $\operatorname{R}(L_{\alpha})$:
  \begin{equation}\label{eq:raIaSt1}
  r_{\alpha} \left( I_{\alpha} \left(1/2, \delta (\chi) \right) \right) = \nu^{1/2} \delta (\chi) +\nu^{-1/2} \delta (\chi) + I^{\alpha} (\nu  \otimes  \chi)+I^{\alpha} (\nu \chi \otimes \nu^{-1} ),
\end{equation}

\begin{equation}\label{eq:raIaTriv1}
  r_{\alpha} \left( I_{\alpha}(1/2, \chi \circ \det) \right) = \nu^{1/2}  \chi \circ \det +\nu^{-1/2}\chi \circ \det + I^{\alpha} ( \chi \otimes \nu^{-1} )+I^{\alpha} (\nu   \otimes \nu^{-1} \chi),
\end{equation}

\begin{equation}\label{eq:raIbSt1}
   \begin{aligned}
     r_{\alpha} \left( I_{\beta}\left(1/2, \delta (\chi) \right) \right) &= I^{\alpha} ( \chi \otimes \nu) + I^{\alpha} (\nu \otimes \nu^{-1} \chi) + I^{\alpha} (\nu \chi \otimes  \chi) \\
     & =\nu^{1/2} \delta (\chi) +\nu^{1/2} \chi \circ  \det + I^{\alpha} ( \chi \otimes \nu) + I^{\alpha} (\nu \otimes \nu^{-1} \chi),
      \end{aligned}
\end{equation}
\begin{equation}\label{eq:raIbTriv1}
   \begin{aligned}
     r_{\alpha} \left(I_{\beta}(1/2, \chi \circ \det) \right) & = I^{\alpha} (\nu \chi \otimes \nu^{-1}) + I^{\alpha} (\nu^{-1} \otimes  \chi) + I^{\alpha} ( \chi \otimes \nu^{-1} \chi) \\
     & = \nu^{-1/2} \delta (\chi) + \nu^{-1/2} \chi \circ \det+ I^{\alpha} (\nu \chi \otimes \nu^{-1}) + I^{\alpha} (\nu^{-1} \otimes  \chi) .
     \end{aligned}
\end{equation}

\end{cor}

\begin{cor}\label{cor:rb1}
  If we take $s=1/2$, $\chi^2 =1$ in lemma \ref{lem:rb}, we
  have in $\operatorname{R}(L_{\beta})$:
  \begin{equation}\label{eq:rbIbSt1}
    \begin{aligned}
 r_{\beta} \left(I_{\beta} \left(1/2, \delta (\chi) \right) \right) = \nu^{1/2} \delta (\chi) + \nu^{-1/2} \delta (\chi) + I^{\beta} (\nu \otimes  \chi) + I^{\beta} ( \chi \otimes \nu \chi),
\end{aligned}
\end{equation}

  \begin{equation}\label{eq:rbIbTriv1}
    \begin{aligned}
    r_{\beta} (I_{\beta}(1/2,  \chi \circ \det))  = \nu^{1/2} \chi \circ \det + \nu^{-1/2}  \chi \circ \det + I^{\beta} (\nu^{-1} \otimes \nu \chi) + I^{\beta} (\nu^{-1} \chi \otimes  \chi),
      \end{aligned}
\end{equation}

\begin{equation}\label{eq:rbIaSt1}
    \begin{aligned}
      r_{\beta} (I_{\alpha}(1/2, \delta (\chi))) &= I^{\beta} (\nu \chi \otimes  \chi) + I^{\beta} (\nu^{-1}  \otimes \nu \chi) + I^{\beta} ( \chi \otimes \nu )\\
      &= \nu^{1/2} \delta (\chi) + \nu^{1/2} \chi \circ \det+ I^{\beta} (\nu  \otimes  \chi) + I^{\beta} (\nu^{-1}  \otimes \nu \chi),
    \end{aligned}
\end{equation}

 \begin{equation}\label{eq:rbIaTriv1}
   \begin{aligned}
     r_{\beta} \left( I_{\alpha}\left(1/2, \chi \circ \det \right) \right)& = I^{\beta} ( \chi \otimes \nu \chi) +I^{\beta}(\nu^{-1} \otimes \chi) + I^{\beta} (\nu^{-1} \chi \otimes \nu)\\
      &= \nu^{-1/2} \delta (\chi) + \nu^{-1/2} \chi \circ \det+  I^{\beta} ( \chi \otimes \nu \chi) +I^{\beta}(\nu^{-1} \otimes \chi) ,
     \end{aligned}
 \end{equation}

In \refeq{eq:rbIaTriv1} and \refeq{eq:rbIaSt1}, we used \cite[Lemma 5.4 (iii)]{BernsteinDeligneKazhdan1986}: $I^{\beta}(\chi_1 \otimes \chi_2) = I^{\beta}
\left(s_{\beta}(\chi_1 \otimes \chi_2)\right)$ in $\operatorname{R}(L_{\beta})$.

\end{cor}

\begin{prop}[G. Muic Proposition 4.1]
  Suppose that $\chi$ is a character of order $2$. Then we have the following:
  \begin{enumerate}[(1)]
  \item The induced representation $I(\nu \chi \otimes \chi)$ has a unique
    irreducible subrepresentation $\pi (\chi)$. We have $r_{\phi} (\pi (\chi)) =
    \nu \chi \otimes \chi + \nu \otimes \chi + \chi \otimes \nu$. $\pi (\chi)$ is
    square integrable.
  \item In $\operatorname{R}(G_2)$, we have:
    \begin{equation}
          \begin{aligned}
           I_{\alpha} (1/2, \delta (\chi)) & = \pi (\chi) + J_{\alpha} (1/2, \delta
           (\chi)), \\
           I_{\beta} (1/2, \delta (\chi)) & = \pi (\chi) +J_{\beta} (1/2, \delta (\chi)) ,\\
           I_{\alpha} (1/2, \chi \circ \det ) & = J_{\beta} (1,\pi (1, \chi)) +J_{\beta}(1/2, \delta (\chi)) ,\\
           I_{\beta}(1/2, \chi \circ \det ) &= J_{\beta} (1, \pi (1, \chi)) + J_{\alpha} (1/2, \delta (\chi)).
          \end{aligned}
        \end{equation}
  \end{enumerate}
\end{prop}
\begin{proof}

By a theorem of F. Rodier \cite{Rodier1981}, we know $I(\nu \chi  \otimes \chi)$ has length $4$,
multiplicity $1$, thus it contains a unique irreducible subrepresentation,
denoted by $\pi (\chi)$.

We know that
\begin{equation}
  \begin{aligned}
 I_{\beta} (\nu^{1/2} \chi \circ \det) +
  I_{\beta} (\nu^{1/2} \delta (\chi))  & =I (\chi \otimes \nu) = I (\nu \otimes \chi)=I
  (\nu \chi \otimes \chi) \\
  &   =I_{\alpha} (\nu^{1/2} \chi \circ \det) + I_{\alpha}
  (\nu^{1/2} \delta (\chi)) .
  \end{aligned}
\end{equation}

Notice that $\pi (\chi)$ is a (unique) subrepresentation of $I_{\beta}(\nu^{1/2} \delta (\chi))$
  and $I_{\alpha} (\nu^{1/2} \delta (\chi))$ since they are subrepresentations
  of $I(\chi \otimes \nu)$ and $I(\nu \chi \otimes \chi)$ respectively. $r_{\alpha} (\pi (\chi))$ must be contained in the common parts of
  \refeq{eq:raIaSt1} and \refeq{eq:raIbSt1}. Since
  \[
    0 \neq \Hom_{G_2} (\pi (\chi), I (\nu \otimes \chi)) =
  \Hom_{L_{\alpha}}(r_{\alpha} (\pi (\chi)), I^{\alpha} (\nu \otimes \chi) ),
\]
 we know $r_{\alpha} (\pi (\chi))$ must contain $I^{\alpha} (\nu \otimes \chi)$. Repeat the above argument with
  $I(\nu \otimes \chi)$ replaced by $I_{\alpha} (\nu^{1/2} \delta (\chi))$, $r_{\alpha} (\pi (\chi))$ must contain $\nu^{1/2} \delta (\chi) $. To satisfy these three conditions, we
  conclude that
  \[
    r_a (\pi(\chi)) = \nu^{1/2} \delta (\chi) +
    I^{\alpha} (\nu \otimes \chi).
  \]
  Composing with $r^{L_{\alpha}}_{T}$, we prove $(1)$.

  By Langlands quotient theorem applied to the following:
  \[
    I(\chi \nu \otimes
    \chi) = I (\nu \otimes \chi) =I_{\beta} (1,I^{\beta} (1 \otimes \chi)),
  \]
  we  see $J_{\beta} (1, \pi (1, \chi))$ is a composition factor of $I(\chi \nu
  \otimes \chi)$. The same theorem applied to $I_{\alpha} (1/2, \delta (\chi))$
  and $I_{\beta} (1/2, \delta (\chi))$, we find the rest two (it has length $4$)
  composition factors: $J_{\alpha} (1/2, \delta (\chi))$ and $J_{\beta} (1/2,
  \delta (\chi))$.

  We have
  \[
    \pi(\chi) + J_{\alpha}(1/2, \delta (\chi)) \subseteq I_{\alpha} (1/2,
    \delta (\chi)), \  [\operatorname{R}(G_2)]
  \]
  and we can exclude the possibility that $I_{\alpha} (1/2,
  \delta (\chi))$ contains more terms because \refeq{eq:raIaTriv1} is not
  contained in \refeq{eq:raIbTriv1} or \refeq{eq:raIbSt1}. Using the same
  argument for $\beta$ case, we can prove $(2)$.

\end{proof}

Using similar methods of the proposition above, we can find the image of $4$
components under $r_{\alpha}$ and $r_{\beta}$.
We list them in the following propositions:
\begin{prop}\label{prop:rarbAll1}
  We have the following:
  \begin{equation}\label{eq:rarbAll1}
    \begin{centering}
  \begin{aligned}
    r_{\alpha} (\pi (\chi))  = \nu^{1/2} \delta &(\chi) +I^{\alpha} (\nu \otimes \chi), \ r_{\beta} (\pi (\chi)) = \nu^{1/2} \delta (\chi) +I^{\beta} (\nu \otimes \chi),\\
    r_{\alpha} (J_{\alpha} (1/2, \delta (\chi)))& = \nu^{-1/2} \delta (\chi) + I^{\alpha} (\nu \chi \otimes \nu^{-1}), \\
    r_{\beta} (J_{\alpha}(1/2, \delta (\chi))) & = \nu^{1/2} \chi \circ \det    + I^{\beta} (\nu^{-1} \otimes \nu \chi), \\
    r_{\alpha} (J_{\beta} (1/2, \delta (\chi)))  & = \nu^{1/2} \chi \circ \det + I^{\alpha} (\nu \otimes \nu^{-1} \chi), \\
    r_{\beta} (J_{\beta} (1/2, \delta (\chi)))& = \nu^{-1/2} \delta (\chi)  + I^{\beta} (\chi \otimes \nu \chi), \\
    r_{\alpha} (J_{\beta} (1, \pi (1, \chi))) & = \nu^{-1/2} \chi \circ \det + I^{\alpha} (\chi \otimes \nu^{-1}), \\  r_{\beta} (J_{\beta} (1, \pi (1, \chi))) & = \nu^{-1/2} \chi \circ \det  + I^{\beta} (\nu^{-1} \chi \otimes \chi).
  \end{aligned}
  \end{centering}
\end{equation}
\end{prop}

\begin{prop}\label{prop:Duality1}
   We compute the Aubert-Zelevinsky duality of all the irreducible
  representations listed above:
\begin{equation}\label{eq:Duality1}
  \begin{aligned}
    D_{G_2} (\pi (\chi)) = J_{\beta} (1, \pi (1, \chi)), \quad D_{G_2} (J_{\alpha} (1/2, \delta(\chi))) = J_{\beta} (1/2, \delta (\chi))  \ [ \operatorname{R}(G_2)]
  \end{aligned}
\end{equation}
\end{prop}
\begin{proof}
  We now compute
  \begin{equation}\label{eq:Duality1proofpichi1}
  \begin{aligned}
  & D_{G_2} (\pi (\chi)) = I \circ r_T^{G_2} (\pi (\chi)) - I_{\alpha} \circ r_{\alpha}(\pi (\chi)) - I_{\beta} \circ r_{\beta} (\pi (\chi)) + \pi (\chi) \\
   &= I( \nu \chi \otimes \chi) + I(\nu \otimes \chi ) +I( \chi \otimes \nu) -  I_{\alpha} ( \nu^{1/2} \delta (\chi)) -I_{\alpha}(I^{\alpha} (\nu \otimes \chi)) \\
   &- I_{\beta}(\nu^{1/2} \delta (\chi)) -I_{\beta}(I^{\beta} (\nu \otimes \chi)) + \pi (\chi) \\
    &=3 I( \nu \chi \otimes \chi)  -  I_{\alpha} ( \nu^{1/2} \delta (\chi)) -I (\nu \otimes \chi)
    - I_{\beta}(\nu^{1/2} \delta (\chi)) -I (\nu \otimes \chi) + \pi (\chi) \\
    & = I( \nu \chi \otimes \chi)  -  I_{\alpha} ( \nu^{1/2} \delta (\chi))- I_{\beta}(\nu^{1/2} \delta (\chi))+ \pi (\chi) \\
    &= I_{\alpha} (1/2, \chi \circ \det) -J_{\beta} (1/2, \delta (\chi)) \\
    &= J_{\beta} (1, \pi (1, \chi)) .
  \end{aligned}
\end{equation}
We can either deduce the rest of this proposition directly from \cite[Lemma 3.1]{Muic1997}
 or use a similar calculation:
 \begin{equation}\label{eq:Duality1proofJa1}
  \begin{aligned}
    & D_{G_2} (J_{\alpha}(1/2, \delta (\chi))) = I \circ r_{\phi} (J_{\alpha}(1/2, \delta (\chi))) - I_{\alpha} \circ r_{\alpha} (J_{\alpha}(1/2, \delta (\chi))) \\
    &- I_{\beta} \circ r_{\beta} (J_{\alpha}(1/2, \delta (\chi))) + J_{\alpha}(1/2, \delta (\chi))\\
   &= I \circ r_T^{L_{\alpha}} \circ \left( \nu^{-1/2} \delta (\chi) + I^{\alpha} (\nu \chi \otimes \nu^{-1}) \right) - I_{\alpha} \circ  \left( \nu^{-1/2} \delta (\chi) + I^{\alpha} (\nu \chi \otimes \nu^{-1}) \right)\\
   & - I_{\beta} \circ \left( \nu^{1/2} \chi \circ \det + I^{\beta} (\nu^{-1} \otimes \nu \chi) \right) + J_{\alpha}(1/2, \delta (\chi))\\
    &= I(\chi \otimes \nu^{-1} \chi) + I (\nu \chi \otimes \nu^{-1}) +I(\nu^{-1} \otimes \nu \chi  ) - I_{\alpha} (\nu^{-1/2} \delta (\chi)) - I (\nu \chi \otimes \nu^{-1})\\
    & - I_{\beta} (\nu^{1/2} \chi \circ \det ) - I (\nu^{-1} \otimes \nu \chi) + J_{\alpha}(1/2, \delta (\chi)) \\
    &= I(\chi \otimes \nu^{-1} \chi) + I (\nu \chi \otimes \nu^{-1}) +I(\nu^{-1} \otimes \nu \chi  ) - I_{\alpha} (\nu^{-1/2} \delta (\chi)) - I (\nu \chi \otimes \nu^{-1})\\
    & - I_{\beta} (\nu^{1/2} \chi \circ \det ) - I (\nu^{-1} \otimes \nu \chi) + J_{\alpha}(1/2, \delta (\chi)) \\
    &= I(\chi \otimes \nu^{-1} \chi) - I_{\alpha} (\nu^{-1/2} \delta (\chi)) - I_{\beta} (\nu^{1/2} \chi \circ \det )+ J_{\alpha}(1/2, \delta (\chi))\\
    & = I_{\alpha} (1/2, \chi \circ \det) -I_{\beta} (1/2, \chi \circ \det) +J_{\alpha} (1/2, \delta (\chi)) \\
    & = J_{\beta} (1,\pi (1, \chi)) +J_{\beta}(1/2, \delta (\chi)) -\left( J_{\beta} (1, \pi (1, \chi)) + J_{\alpha} (1/2, \delta (\chi)) \right) +J_{\alpha} (1/2, \delta (\chi)) \\
   & = J_{\beta} (1/2, \delta (\chi)).
  \end{aligned}
\end{equation}
In the steps of above computations, we need to use
\cite[Lemma 5.4 (iii)]{BernsteinDeligneKazhdan1986}  to get
\[
  I_{\alpha}
  (\nu^{-1/2} \delta (\chi)) =I_{\alpha} (\nu^{1/2} \delta (\chi))  \]
in $\operatorname{R}(G_2)$, and \cite[Geometric Lemma]{BernsteinZelevinsky1977}
for computing $ r_T^{L_{\alpha}}  I^{\alpha} (\nu \chi \otimes \nu^{-1})$.
\end{proof}

\subsubsection{Case \refeq{Muic1997itm:3}  within the case $\mathfrak{s}=[T,  \xi \otimes \xi]_G$:
  $\chi$ ramified quadratic}

If $\xi_2$ is ramified quadratic, then by the discussion in
\cite[Section 9.3.2]{AubertXu2023explicit}, we have $\mathfrak{J}^{\mathfrak{s}}
=\SO_4(\C)$ and $\H^{\mathfrak{s}} = \H
(\SO_4(F), 1)$, we know that $\SO_4 (\C) \cong \SL_2 (\C) \times \SL_2 (\C) /
\left\{  \pm 1 \right\}$. Using \cite[section 3 table 2.1]{Ram2004Rank2}, we
can index the modules by two triples of type $A_1$:

We get the proposition for modules of Hecke algebra:

\begin{prop}
   Using the indexing triples to denote the standard modules, we have:
  \begin{equation}
    \begin{aligned}
      D_{\H^{\mathfrak{s}}} (M_{(t_a, e_{\alpha_1}, 1), (t_a, e_{\alpha_1}, 1)}) &= M_{(t_a, 0, 1), (t_a, 0, 1)}, \\
      D_{\H^{\mathfrak{s}}} (M_{(t_a, 0, 1), (t_a, e_{\alpha_1}, 1)}) &= M_{(t_a, e_{\alpha_1}, 1), (t_a, 0, 1)} .
    \end{aligned}
  \end{equation}
\end{prop}

\subsubsection{Case \refeq{Muic1997itm:3}  within the case $\mathfrak{s}=[T,  1]_G$: $\chi$ unramified
quadratic}
If $\xi_2$ is unramified quadratic,  $\mathfrak{J}^{\mathfrak{s}}
=G_2(\C)$, and $\H^{\mathfrak{s}} = \H
(G_2(F), 1)$. Then by \cite[table
16]{AubertXu2023explicit}, we get the proposition for modules of Hecke algebra:

\begin{prop}
   Using the indexing triples to denote the standard modules, we have:
  \begin{equation}
    \begin{aligned}
      D_{\H^{\mathfrak{s}}} (M_{t_d, e_{\alpha^{\vee}}+e_{\alpha^{\vee} + 2 \beta^{\vee}}, 1}) &= M_{t_d, 0,1}, \\
      D_{\H^{\mathfrak{s}}} (M_{t_d, e_{\alpha^{\vee}}, 1}) &= M_{t_d, e_{2\beta^{\vee}+\alpha^{\vee}}, 1}  .
    \end{aligned}
  \end{equation}
\end{prop}

\subsection{Case \refeq{Muic1997itm:3} computations for $s=1/2$ and $\chi^3=1$, $\chi \neq 1$}

If $s=1/2$ and $\chi^3=1$, then by \cite[Theorem 3.1]{Muic1997}, $I_{\alpha} (s,
\delta (\chi))$ and $I_{\alpha} (s, \chi \circ \det)$ reduce, while $I_{\beta} (s,
\delta (\chi))$ or $I_{\beta} (s, \chi \circ \det)$ do not.

If $s=1/2$ and  $\chi^3=1$, then
\begin{equation*}
  \begin{aligned}
 I(\nu \chi \otimes \chi) =
I (\nu \chi^{-1} \otimes  \nu^{-1} \chi^{-1})  & = I (\chi \otimes \chi \nu^{-1})=
I(\nu^{-1} \chi^{-1} \otimes \chi^{-1})  \\
  &   = I (\nu \chi^{-1} \otimes \chi^{-1})  =
I(\nu \chi \otimes \nu^{-1} \chi)  \ [\operatorname{R}(G_2)]
  \end{aligned}
\end{equation*}
(we omit the terms differ by $s_{\alpha}$). We deduce
\begin{equation*}
  \begin{aligned}
I_{\alpha} (\nu^{-1/2} \chi^{-1} \circ \det) + I_{\alpha}
  (\nu^{-1/2} \delta (\chi^{-1})) & =I_{\alpha} (\nu^{1/2} \chi^{-1} \circ \det) + I_{\alpha}
  (\nu^{1/2} \delta (\chi^{-1}))  \\
  &   =I_{\alpha} (\nu^{1/2} \chi \circ \det) + I_{\alpha}
  (\nu^{1/2} \delta (\chi)) \ [\operatorname{R}(G_2)]
  \end{aligned}
\end{equation*}
from
\[
  I(\nu^{-1} \chi^{-1} \otimes \chi^{-1}) =
  I(\nu \chi^{-1} \otimes \chi^{-1}) = I(\nu \chi \otimes \chi) \
  [\operatorname{R}(G_2)]
\]
respectively.

Moreover, since
\[
  I^{\beta} (\nu \chi  \otimes \chi) = I^{\beta} \left(s_{\beta} \circ
(\nu \chi
\otimes \chi) \right) =I^{\beta} (\nu \chi^{-1} \otimes \chi^{-1}) \
[\operatorname{R}(L_{\beta})]
\]
and they are
irreducible, they are thus isomorphic as representations. We have
\[
  I(\nu \chi \otimes
\chi) = I_{\beta}(I^{\beta} (\nu \chi \otimes \chi)) \cong I_{\beta} \left(I^{\beta} (\nu
  \chi^{-1} \otimes \chi^{-1}) \right) =I(\nu \chi^{-1} \otimes \chi^{-1})
\]
as representations.

\begin{cor}\label{cor:ra2}
  If we take $\chi^3=1$, $s=1/2$ in  Lemma \ref{lem:ra}, we have:
 \begin{equation}\label{eq:raIaSt2}
    \begin{aligned}
      r_{\alpha} \left( I_{\alpha} \left( 1/2, \delta \left(\chi \right) \right) \right) & = \nu^{1/2} \delta (\chi) +\nu^{-1/2} \delta (\chi^{-1}) + I^{\alpha} (\nu \chi^{-1} \otimes \chi^{-1})+I^{\alpha} (\nu \chi \otimes \nu^{-1} \chi) \\
      &=  \nu^{1/2} \delta (\chi) +\nu^{-1/2} \delta (\chi^{-1}) + \nu^{1/2} \chi^{-1} \circ \det + \nu^{1/2} \delta (\chi^{-1}) \\ & +I^{\alpha} (\nu \chi \otimes \nu^{-1} \chi),
       \end{aligned}
\end{equation}
\begin{equation}\label{eq:raIaTriv2}
   \begin{aligned}
     r_{\alpha} \left( I_{\alpha}(1/2, \chi \circ \det) \right) &= \nu^{1/2}  \chi \circ \det +\nu^{-1/2}\chi^{-1} \circ \det + I^{\alpha} ( \chi \otimes \nu^{-1} \chi)+I^{\alpha} (\nu \chi^{-1}  \otimes \nu^{-1} \chi^{-1})\\
                                                                &=  \nu^{1/2}  \chi \circ \det +\nu^{-1/2}\chi^{-1} \circ \det + \nu^{-1/2} \delta (\chi) + \nu^{-1/2} \chi \circ \det \\
     &+I^{\alpha} (\nu \chi^{-1}  \otimes \nu^{-1} \chi^{-1}).
     \end{aligned}
\end{equation}
If we replace $\chi$ by $\chi^{-1}$, then we have:
\begin{equation}\label{eq:raIaSt22}
    \begin{aligned}
      r_{\alpha} \left( I_{\alpha} \left( 1/2, \delta (\chi^{-1}) \right) \right) & =  \nu^{1/2} \delta (\chi^{-1}) +\nu^{-1/2} \delta (\chi) + \nu^{1/2} \chi \circ \det + \nu^{1/2} \delta (\chi) \\
      &+I^{\alpha} (\nu \chi^{-1} \otimes \nu^{-1} \chi^{-1}),
       \end{aligned}
\end{equation}
\begin{equation}\label{eq:raIaTriv22}
   \begin{aligned}
     r_{\alpha} \left( I_{\alpha} \left( 1/2, \chi^{-1} \circ \det \right) \right) &=  \nu^{1/2}  \chi^{-1} \circ \det +\nu^{-1/2}\chi \circ \det + \nu^{-1/2} \delta (\chi^{-1}) + \nu^{-1/2} \chi^{-1} \circ \det \\
     &+I^{\alpha} (\nu \chi  \otimes \nu^{-1} \chi) .
     \end{aligned}
\end{equation}
\end{cor}

\begin{cor}\label{cor:rb2}
   If we take $\chi^3=1$, $s=1/2$ in \refeq{eq:rbIaSt} and \refeq{eq:rbIaTriv} of Lemma \ref{lem:rb}, we have:
 \begin{equation}\label{eq:rbIaSt2}
  r_{\beta} (I_{\alpha}(1/2, \delta (\chi))) = I^{\beta} (\nu \chi \otimes  \chi) + I^{\beta} (\nu^{-1} \chi \otimes \nu \chi) + I^{\beta} ( \chi^{-1} \otimes \nu \chi^{-1}),
\end{equation}
\begin{equation}\label{eq:rbIaTriv2}
  r_{\beta} \left( I_{\alpha}(1/2, \chi \circ \det) \right) = I^{\beta} ( \chi \otimes \nu \chi) +I^{\beta}(\nu^{-1} \chi \otimes  \chi) + I^{\beta} (\nu^{-1} \chi^{-1} \otimes \nu \chi^{-1}) .
\end{equation}
If we replace $\chi$ by $\chi^{-1}$, then:
 \begin{equation}\label{eq:rbIaSt22}
  r_{\beta} \left( I_{\alpha} \left( 1/2, \delta (\chi^{-1}) \right) \right) = I^{\beta} (\nu \chi \otimes  \chi) + I^{\beta} (\nu^{-1} \chi^{-1} \otimes \nu \chi^{-1}) + I^{\beta} ( \chi \otimes \nu \chi),
\end{equation}
\begin{equation}\label{eq:rbIaTriv22}
  r_{\beta} \left( I_{\alpha}\left( 1/2, \chi^{-1} \circ \det \right) \right) = I^{\beta} ( \chi^{-1} \otimes \nu \chi^{-1}) +I^{\beta}(\nu^{-1} \chi \otimes \chi) + I^{\beta} (\nu^{-1} \chi \otimes \nu \chi) .
\end{equation}

In equation \refeq{eq:rbIaSt22} and \refeq{eq:rbIaTriv22} we used $I^{\beta} \left(s_{\beta} (\chi_1
  \otimes \chi_2) \right) = I^{\beta} (\chi_1 \otimes \chi_2)$ to the first term
and second term respectively.
\end{cor}
\begin{prop}[G. Muic Proposition 4.2]\label{prop:Muic4.2}
  If a character $\chi$ satisfies $\chi^3=1$, then we have the following
  \begin{enumerate}
  \item The induced representation $I(\nu \chi \otimes \chi)$ has a unique
    subrepresentation $\pi (\chi)$. We have
    \[
      r_{\emptyset} (\pi (\chi)) = \nu
      \chi \otimes \chi + \nu \chi^{-1} \otimes \chi^{-1},
    \]
    and $\pi (\chi) \cong
    \pi (\chi^{-1})$.
  \item In $\operatorname{R}(G_2)$, we have:
 \begin{equation}\label{eq:Ia2}
   \begin{aligned}
     I_{\alpha} \left( 1/2, \delta (\chi) \right) &= \pi (\chi) +J_{\alpha} \left( 1/2, \delta (\chi) \right), \\
     I_{\alpha}  \left( 1/2, \chi \circ \det \right) & = J_{\beta} \left( 1, \pi (\chi^{-1}, \chi^{-1}) \right) +J_{\alpha} \left( 1/2, \delta (\chi^{-1}) \right), \\
 I_{\alpha} \left( 1/2, \delta (\chi^{-1}) \right) &= \pi (\chi) +J_{\alpha}(1/2, \delta (\chi^{-1})), \\
     I_{\alpha} \left( 1/2, \chi^{-1} \circ \det \right) & = J_{\beta} \left(1, \pi (\chi^{-1}, \chi^{-1} ) \right) +J_{\alpha} \left(1/2, \delta (\chi)\right)  .
   \end{aligned}
 \end{equation}
  \end{enumerate}
\end{prop}

\begin{rem}
  We point out that there is a typo in \cite[Proposition 4.2]{Muic1997}, the
  $J_{\beta} \left( 1, \pi (\chi, \chi^{-1}) \right)$ should be replaced by
  $J_{\beta} \left( 1,
  \pi (\chi^{-1}, \chi^{-1}) \right)$ instead.
\end{rem}
\begin{proof}
  By a theorem of F. Rodier \cite{Rodier1981}, we know $I( \nu \chi \otimes \chi)$ has length $4$,
multiplicity $1$, thus it contains a unique irreducible subrepresentation,
denoted by $\pi (\chi)$.

Notice that $\pi (\chi)$ is a (unique) subrepresentation of $I_{\alpha}(\nu^{1/2} \delta (\chi^{-1}))$
  and $I_{\alpha} \left( \nu^{1/2} \delta (\chi) \right)$ since they are subrepresentations
  of $I(\nu \chi^{-1} \otimes \chi^{-1})$ and $I(\nu \chi \otimes \chi)$
  respectively. $r_{\alpha} (\pi (\chi))$ must be contained in the common parts of
  \refeq{eq:raIaSt2} and \refeq{eq:raIaSt22}. Since
  \[
    0 \neq \Hom_{G_2} (\pi (\chi), I_{\alpha}(\nu^{1/2} \delta (\chi^{-1}))) =
    \Hom_{L_{\alpha}}(r_{\alpha} (\pi (\chi)),\nu^{1/2} \delta (\chi^{-1})) ,
    \]
 we know  $r_{\alpha} (\pi (\chi))$ must contain $\nu^{1/2} \delta (\chi^{-1})$. Repeat the above argument with
  $I(\nu \chi \otimes \chi)$ replaced by $I( \nu \chi^{-1} \otimes \chi^{-1})$, $r_{\alpha} (\pi (\chi))$ must contain $\nu^{1/2} \delta (\chi)$. To satisfy these three conditions, we
  conclude that
  \[
    r_{\alpha}(\pi (\chi)) = \nu^{1/2} \delta (\chi) + \nu^{1/2}
    \delta (\chi^{-1}).
  \]
  Composing with $r^{L_{\alpha}}_{T}$, we prove
  $(1)$.

Applying Langlands quotient theorem to the following:
  \[
    I(\nu \chi \otimes
  \chi) =I(\nu \chi^{-1} \otimes \chi^{-1}) = I_{\beta} (1, \pi (\chi^{-1},
  \chi^{-1})),
\]
we see $J_{\beta} (1, \pi (\chi^{-1}, \chi^{-1}))$ is a composition factor of
  $I(\nu \chi
  \otimes \chi)$. The same theorem applied to $I_{\alpha} (1/2, \delta (\chi))$
  and $I_{\alpha} (1/2, \delta (\chi^{-1}))$, we find the rest two (it has length $4$)
  composition factors: $J_{\alpha} (1/2, \delta (\chi))$ and $J_{\alpha} (1/2,
  \delta (\chi^{-1}))$.

   We have $\pi(\chi) + J_{\alpha}(1/2, \delta (\chi)) \subseteq I_{\alpha} (1/2,
  \delta (\chi))$, and we can exclude the possibility that $I_{\alpha} (1/2,
  \delta (\chi))$ contains more terms because \refeq{eq:raIaTriv2} is not
  contained in \refeq{eq:raIaTriv22} or \refeq{eq:raIaSt22}. Using the same
  argument for $\chi^{-1}$ case, we can prove $(2)$.

\end{proof}

Using similar methods of the proposition above, we can find the image of $4$
components under $r_{\alpha}$ and $r_{\beta}$.
We list them in the following proposition:
\begin{prop}\label{prop:rarbAll2}
   Based on the proof of Proposition \ref{prop:Muic4.2}, we have the following
  \begin{equation}\label{eq:raraAll2}
  \begin{aligned}
    r_{\alpha} (\pi (\chi))&= \nu^{1/2} \delta (\chi) + \nu^{1/2}  \delta (\chi^{-1}), \quad r_{\beta} (\pi (\chi)) = I^{\beta} (\nu \chi \otimes \chi),\\
    r_{\alpha} (J_{\alpha} (1/2, \delta (\chi))) &= \nu^{-1/2} \delta (\chi^{-1}) + \nu^{1/2} \chi^{-1} \circ \det +I^{\alpha} (\nu \chi \otimes \nu^{-1} \chi), \\
    r_{\beta} (J_{\alpha}(1/2, \delta (\chi))) &= I^{\beta} (\nu^{-1} \chi \otimes \nu \chi) +I^{\beta} (\chi^{-1} \otimes \nu \chi^{-1}), \\
    r_{\alpha} (J_{\alpha} (1/2, \delta (\chi^{-1}))) & = \nu^{-1/2} \delta (\chi)+ \nu^{1/2} \chi \circ \det +I^{\alpha} (\nu \chi^{-1} \otimes \nu^{-1} \chi^{-1}) , \\
    r_{\beta} (J_{\alpha} (1/2, \delta (\chi^{-1}))) &= I^{\beta} (\nu^{-1} \chi^{-1} \otimes \nu \chi^{-1}) +I^{\beta} (\chi \otimes \nu \chi), \\
    r_{\alpha} (J_{\beta} (1, \pi (\chi^{-1}, \chi^{-1}))) &= \nu^{-1/2} \chi^{-1} \circ \det + \nu^{-1/2} \chi \circ \det, \\
    r_{\beta} (J_{\beta} (1, \pi (\chi^{-1}, \chi^{-1}))) & = I^{\beta} (\nu^{-1} \chi \otimes \chi).
  \end{aligned}
\end{equation}
\end{prop}

\begin{prop}\label{prop:Duality2}
   We compute the Aubert-Zelevinsky duality of all the irreducible
  representations listed above:
\begin{equation}\label{eq:Duality2}
  \begin{aligned}
    D_{G_2} (\pi (\chi)) = J_{\beta} (1, \pi (\chi^{-1}, \chi^{-1})), \quad D_{G_2} (J_{\alpha} (1/2, \delta(\chi))) = J_{\alpha} (1/2, \delta (\chi^{-1}))  \ [\operatorname{R}(G_2)]
  \end{aligned}
\end{equation}
\end{prop}
\begin{proof}
  We compute as follows:
  \begin{equation}\label{eq:Duality1proofpichi2}
  \begin{aligned}
  & D_{G_2} (\pi (\chi)) = I \circ r_T^{G_2} (\pi (\chi)) - I_{\alpha} \circ r_{\alpha} (\pi (\chi)) - I_{\beta} \circ r_{\beta} (\pi (\chi)) + \pi (\chi) \\
   &= I(\nu
     \chi \otimes \chi) +I( \nu \chi^{-1} \otimes \chi^{-1})   - I_{\alpha}\left(\nu^{1/2} \delta (\chi) \right) - I_{\alpha} \left( \nu^{1/2}  \delta (\chi^{-1})  \right) \\
  & -I_{\beta} \left(  I^{\beta} (\nu \chi \otimes \chi)\right) + \pi (\chi)\\
   & = I_{\alpha} (1/2, \chi \circ \det) -I_{\alpha} (1/2, \delta (\chi^{-1})) + \pi (\chi) \\
   & = J_{\beta} (1, \pi (\chi^{-1}, \chi^{-1})) +J_{\alpha} (1/2, \delta (\chi^{-1})) -\pi (\chi) -J_{\alpha}(1/2, \delta (\chi^{-1})) + \pi (\chi) \\
   &= J_{\beta} (1, \pi (\chi^{-1}, \chi^{-1})) .
  \end{aligned}
\end{equation}
Similarly, we have
 \begin{equation}\label{eq:Duality1proofJa2}
  \begin{aligned}
    & D_{G_2} \left( J_{\alpha} \left(1/2, \delta (\chi) \right) \right) = I \circ r_T^{G_2}  \left(J_{\alpha} \left(1/2, \delta (\chi ) \right) \right) - I_{\alpha} \circ r_{\alpha} \left( J_{\alpha} \left(1/2, \delta (\chi) \right) \right)\\
    & - I_{\beta} \circ r_{\beta} \left( J_{\alpha} \left(1/2, \delta (\chi) \right) \right) + J_{\alpha} \left(1/2, \delta (\chi) \right) \\
    &= I \circ r_T^{L_{\alpha}} ( \nu^{-1/2} \delta (\chi^{-1}) ) + I \circ r_T^{L_{\alpha}} ( \nu^{1/2} \chi^{-1} \circ \det) +I \circ   r_T^{L_{\alpha}} (I^{\alpha} (\nu \chi \otimes \nu^{-1} \chi))\\
    & - I_{\alpha} (\nu^{-1/2} \delta (\chi^{-1})) -I_{\alpha} (\nu^{1/2} \chi^{-1} \circ \det ) -I_{\alpha} ( I^{\alpha} (\nu \chi \otimes \nu^{-1} \chi))) - I_{\beta} ( I^{\beta} (\nu^{-1} \chi \otimes \nu \chi) ) \\ & -I_{\beta} ( I^{\beta} (\chi^{-1} \otimes \nu \chi^{-1}) ) + J_{\alpha} \left(1/2, \delta (\chi) \right) \\
    & = I \left( \chi^{-1} \otimes \nu^{-1} \chi^{-1} \right) + I \left( \chi^{-1} \otimes \nu \chi^{-1} \right)  + 2 I \left( \nu \chi \otimes \nu^{-1} \chi \right)  - I_{\alpha} (1/2,  \delta (\chi)) \\ &-I_{\alpha} (1/2 , \chi^{-1} \circ \det ) -I(\nu \chi \otimes \nu^{-1} \chi) - I( \nu^{-1}\chi \otimes \nu \chi)  -I (\chi^{-1} \otimes \nu \chi^{-1})\\ & + J_{\alpha} \left(1/2, \delta (\chi) \right) \\
    &=  I \left( \chi^{-1} \otimes \nu^{-1} \chi^{-1} \right)  - I_{\alpha} (1/2,  \delta (\chi)) -I_{\alpha} (1/2 , \chi^{-1} \circ \det ) +J_{\alpha} \left(1/2, \delta (\chi) \right)  \\
    & = J_{\beta} (1, \pi (\chi^{-1}, \chi^{-1})) +J_{\alpha} (1/2, \delta (\chi^{-1})) -J_{\beta} \left(1, \pi (\chi^{-1}, \chi^{-1}\right)) -J_{\alpha} \left(1/2, \delta (\chi)\right) \\
    & + J_{\alpha} \left(1/2, \delta (\chi) \right)  \\
    & = J_{\alpha} (1/2, \delta (\chi^{-1})).
  \end{aligned}
\end{equation}
\end{proof}

\subsubsection{Case \refeq{Muic1997itm:3}  within the case $\mathfrak{s}=[T,  \xi \otimes \xi]_G$:
  $\chi$ ramified cubic}

If $\chi$ is ramified cubic, $\mathfrak{J}^{\mathfrak{s}}
=\SL_3(\C)$, and $\H^{\mathfrak{s}} = \H
(\operatorname{PGL}_3(F), 1)$. By \cite[table
20]{AubertXu2023explicit} and \cite[table 4.1]{Ram2004Rank2}, we get the proposition for modules of Hecke algebra:

\begin{prop}
  Using
the indexing triples to denote the standard modules, we have:
  \begin{equation}
    \begin{aligned}
      D_{\H^{\mathfrak{s}}} (M_{t_a, e_{\alpha^{\vee}}+e_{2 \alpha^{\vee} + 3 \beta^{\vee}}, 1}) &= M_{t_a, 0,1}, \\
      D_{\H^{\mathfrak{s}}} (M_{t_a, e_{\alpha^{\vee}}, 1}) &= M_{t_a, e_{3 \beta^{\vee}+2 \alpha^{\vee}}, 1}  .
    \end{aligned}
  \end{equation}
\end{prop}

\subsubsection{Case \refeq{Muic1997itm:3}  within the case $\mathfrak{s}=[T,  1]_G$:
  $\chi$ unramified cubic}\label{subsec:case3last}

If $\chi$ is unramified cubic, $\mathfrak{J}^{\mathfrak{s}}
=G_2(\C)$, and $\H^{\mathfrak{s}} = \H
(G_2(F), 1)$ by \cite[table
19]{AubertXu2023explicit} and \cite[table 6.2]{Ram2004Rank2}. We get the proposition for modules of Hecke algebra:

\begin{prop}
   Using the indexing triples to denote the standard modules, we have:
  \begin{equation}
    \begin{aligned}
      D_{\H^{\mathfrak{s}}} (M_{t_c, e_{\alpha^{\vee}}+e_{ \alpha^{\vee} + 3 \beta^{\vee}}, 1}) &= M_{t_c, 0,1}, \\
      D_{\H^{\mathfrak{s}}} (M_{t_c, e_{\alpha^{\vee}}, 1}) &= M_{t_c, e_{3 \beta^{\vee}+ \alpha^{\vee}}, 1}  .
    \end{aligned}
  \end{equation}
\end{prop}

\section{Verification of several cases of the Bernstein conjecture for $G_2$}

Let $\pi^{+}$ denote the Hermitian contragredient (taking contragredient and complex conjugate) of a smooth finite length representation $\pi$. If $\pi \cong \pi^{+}$, we say that $\pi$ is a \emph{Hermitian} representation. Or equivalently, if there is a non–degenerate $G$–invariant Hermitian form $\langle \cdot, \cdot \rangle$ on $V_{\pi}$. An Hermitian representation $\pi$ is \emph{unitarizable} if the form $\langle \cdot, \cdot \rangle$ is definite. We recall some basic facts about unitarizable representations, see \cite{Casselman1995} and \cite{MuicTadic2011}.

\begin{prop}\label{prop:unitarizable}
  \begin{enumerate}
  \item A unitarizable admissible representation of $G$  is  isomorphic to a direct sum of irreducible admissible unitarizable representations, each isomorphism class occurring with finite multiplicity.
  \item An irreducible admissible square-integrable representation is unitarizable.
  \item Let $P = MN$ be a parabolic subgroup of $G$ and $\sigma$ be an irreducible representation of $M$, then the representation $i_{P}^{G}(\sigma)$ is unitarizable if $\sigma$ is.
  \end{enumerate}
\end{prop}
We can verify many cases of the following conjecture based on our computations in Section \ref{sec:computationsG2} and \cite[Section 5,6]{Muic1997}.
\begin{conj}[Bernstein]
If $\pi$ is an irreducible unitarizable representation of $G_2$ then its Aubert-Zelevinsky duality $\ND (\pi)$ is also unitarizable.
\end{conj}
\begin{proof}
  The representation $\pi (\chi)$ in Case $(3)$ with $\chi$ ramified quadratic $\mathfrak{J}^{\mathfrak{s}} = \SO_4 (\mathbb{C})$, and the representation $\pi (\chi)$ in Case $(3)$ with $\chi$ ramified cubic $\mathfrak{J}^{\mathfrak{s}} = \SL_3 (\mathbb{C})$ in the block with $\mathfrak{s} = [T , \xi \otimes \xi] $; the representations $\pi (1)$ and $\pi' (1)$ in Case $(2)$ with $\chi =1$, $s =1/2$, $\mathfrak{J}^{\mathfrak{s}}  = G_2 (\mathbb{C})$, the representation $\St_{G_2}$ in Case $(2)$ with $s = 3/2$ and $\chi =1$, the representation $\pi (\chi)$ in Case $(3)$ with $\chi$ unramified quadratic $\mathfrak{J}^{\mathfrak{s}} = G_2 (\mathbb{C})$ and the representation $\pi (\chi)$ in Case $(3)$ with $s = 1/2$ and $\chi$ cubic $\mathfrak{J}^{\mathfrak{s}} = G_2 (\mathbb{C})$ in the block indexed by $\mathfrak{s} = [T,1]$, are unitarizable since they are square-integrable. Their images under the Aubert-Zelevinsky duality $D_{G_2}$ are $J_{\beta}(1, \pi (1, \chi))$ (\cite[Theorem 5.1 (ii)5]{Muic1997}),  $J_{\beta} (1, \pi (\chi^{-1}, \chi^{-1}))$ (\cite[Section 6]{Muic1997}), $J_{\alpha} (1/2, \delta (1))$ (\cite[Theorem 5.1 (ii)5]{Muic1997}), $J_{\beta} (1, \pi (1,1))$ (\cite[Corollary 5.1]{Muic1997}), $1_{G_2}$, $J_{\beta} (1, \pi (1, \chi))$ and $J_{\beta} (1, \pi (\chi^{-1}, \chi^{-1}))$ respectively, and are all unitarizable representations due to the reasons in the brackets.

  The representation $J_{\alpha} (1/2, \delta (\chi))$ in Case $(3)$ with $\chi$ ramified quadratic $\mathfrak{J}^{\mathfrak{s}} = \SO_4 (\mathbb{C})$, $\mathfrak{s} = [T, \xi \otimes \xi]$ is sent to $J_{\beta}(1/2, \delta(\chi))$ by  $D_{G_2}$. The representation $J_{\beta} (1/2, \delta (1))$ in Case $(2)$ with $\chi = 1$, $s = 1/2$, $\mathfrak{J}^{\mathfrak{s}} = G_{2} (\mathbb{C})$, $\mathfrak{s} = [T, 1]$ is mapped to itself. The representation $J_{\alpha} (1/2, \delta (\chi))$ in Case $(3)$ with $\chi$ unramified quadratic $\mathfrak{J}^{\mathfrak{s}} = G_{2} (\mathbb{C})$, $\mathfrak{s} = [T, 1]$ is mapped to $J_{\beta}(1/2, \delta (\chi))$. All these $6$ representations are unitarizable by \cite[Theorem 5.1 (i)]{Muic1997}.

  The following representations are not unitarizable. In the Bernstein block
  indexed by $\mathfrak{s}= [T, \xi \otimes \xi]$: $J_{\alpha} (1/2, \delta
  (\chi))$ and $D_{G_2}(J_{\alpha} (1/2, \delta
  (\chi))) = J_{\alpha} (1/2, \delta
  (\chi^{-1}))$ in Case $(3)$ with $\chi$ ramified cubic
  $\mathfrak{J}^{\mathfrak{s}} = \SL_3 (\mathbb{C})$ by \cite[Theorem 5.1 (i)]{Muic1997}.  In the Bernstein block
  indexed by $\mathfrak{s}= [T, 1]$: the representation $J_{\alpha} (3/2, \delta
  (1))$ and $D_{G_2}(J_{\alpha} (3/2, \delta
  (1))) = J_{\beta} (5/2, \delta
  (1))$ in Case $(2)$ with $s = 3/2$ and $\chi =1$ by \cite[Theorem 5.1 (i)]{Muic1997};  the representation
  $I_{\alpha} (\delta (\nu^{\pm 1/2}\xi_{2}))$ and $D_{G_2}(I_{\alpha} (\delta
  (\nu^{\pm 1/2}\xi_{2}))) = I_{\alpha} (\nu^{\pm 1/2}\xi_{2} \circ \det )$ in
  Case $(3)$ with $\xi_{2}$ unramified and  $\mathfrak{J}^{\mathfrak{s}} = G_2
  (\mathbb{C})$ by Proposition \ref{prop:unitarizable} $(3)$; $J_{\alpha} (1/2, \delta
  (\chi))$ and $D_{G_2}(J_{\alpha} (1/2, \delta
  (\chi))) = J_{\alpha} (1/2, \delta
  (\chi^{-1}))$ in Case $(3)$ with $s =1/2$, $\chi $ cubic
  $\mathfrak{J}^{\mathfrak{s}} = G_2 (\mathbb{C})$ by \cite[Theorem 5.1 (i)]{Muic1997}.
\end{proof}

\newcommand{\etalchar}[1]{$^{#1}$}

\end{document}